\documentclass[11pt,a4paper]{amsart}

\usepackage[utf8]{inputenc}
\usepackage[english]{babel}
\usepackage{amsmath,amscd,amssymb}
\usepackage{amsthm}
\usepackage{graphicx}
\usepackage{xcolor}
\usepackage{textgreek}

\usepackage[pdfencoding=auto]{hyperref}  
 \hypersetup{
     colorlinks,
     linkcolor= {blue!80!black},
     citecolor={blue!80!black},
     urlcolor={blue!80!black}}
\usepackage{mathrsfs}
\usepackage{parskip}
\usepackage[left=3cm,right=3cm,top=4.5cm,bottom=4.5cm]{geometry}
\usepackage{enumitem}
\usepackage[nameinlink]{cleveref}
\usepackage{verbatim}

\newtheorem{main-theorem}{Theorem}
\newtheorem{proposition}{Proposition}[section]

\newtheorem{lemma}[proposition]{Lemma}

\theoremstyle{remark}

\numberwithin{equation}{section}

\newcommand{\N}{\mathbb{N}}   
\newcommand{\R}{\mathbb{R}}

\newcommand{\IC}{\mathbb{C}} 
\newcommand{\IS}{\mathbb{S}}

\newcommand{\cG}{\mathcal{G}}

\def\SS{\mathbb{S}} 

\def\cQ{\mathcal{Q}} %

\newcommand{\norm}[1]{\lVert #1 \rVert}

\newcommand{\gae}{\gamma_{\mathrm{exp}}} 
\newcommand{\qe}{q_{\mathrm{exp}}} 

\newcommand{\qexp}[1]{\left[ #1 \right]_{\mathrm{exp}}}
\newcommand{\gexp}[1]{\left[ #1 \right]_{\mathrm{exp}}}

\newcounter{sidenote}
\DeclareMathOperator{\Spec}{\mathrm{Spec}}

\let\Im\relax
\DeclareMathOperator{\Im}{\mathrm{Im}}

\title[The Born approximation in the Calderón problem II]{The Born approximation in the three-dimensional Calderón problem II: Numerical reconstruction in the radial case}

\author[J. A. Barceló]{Juan A. Barceló} 

\address[JAB]{M$^2$ASAI. Universidad Politécnica de Madrid, ETSI Caminos, c. Profesor Aranguren s/n, 28040, Madrid, Spain.}
\email{juanantonio.barcelo@upm.es}

\author[C. Castro]{Carlos Castro} 
\address[CC]{M$^2$ASAI. Universidad Politécnica de Madrid, ETSI Caminos, c. Profesor Aranguren s/n, 28040, Madrid, Spain.}
\email{carlos.castro@upm.es}

\author[F. Macià]{Fabricio Macià}
\address[FM]{M$^2$ASAI. Universidad Politécnica de Madrid, ETSI Navales, Avda. de la Memoria, 4, 28040, Madrid, Spain.}
\email{fabricio.macia@upm.es}

\author[C. J. Meroño]{Cristóbal J. Meroño}
\address[CJM]{M$^2$ASAI. Universidad Politécnica de Madrid, ETSI Caminos, c. Profesor Aranguren s/n, 28040, Madrid, Spain.}
\email{cj.merono@upm.es}

\begin{document}

\begin{abstract}
  In this work we illustrate  a  number of properties of the Born approximation  in the three-dimensional Calderón inverse conductivity problem by numerical experiments. The results are  based on an  explicit representation formula for the Born approximation recently introduced by the authors. We focus on the particular  case of radial conductivities in the ball $B_R \subset \mathbb{R}^3 $ of radius $R$, in which the linearization of the Calderon problem is equivalent to a Hausdorff moment problem. We give numerical evidences that the Born approximation is well defined for $L^{\infty}$ conductivities, and present a novel numerical algorithm to reconstruct a radial conductivity from the Born approximation under a suitable smallness assumption.  
We also show that the Born approximation has  depth-dependent uniqueness and approximation capabilities depending on the distance (depth) to the boundary $\partial B_R$. 
We then investigate how increasing the radius $R$ affects the quality of the Born approximation, and the existence of a scattering limit as $R\to \infty$. Similar properties are also illustrated in the inverse boundary problem for the Schrödinger operator $-\Delta +q$, and strong recovery of singularity results are observed in this case. 
\end{abstract}

\maketitle

\section{Introduction} \label{sec:1_background}

Let $B_R$ be the ball of radius $0<R<\infty$ in $\mathbb{R}^d$ with $d\ge 3$, and  $\gamma \in L^\infty(B_R, \R_+) $ a bounded real non-negative conductivity such that $\gamma(x)\ge c >0 $ holds $a.e.$ in $ B_R$ for some constant $c>0$. For a given $\gamma$ we consider the so-called  Dirichlet to Neumann (DtN) map  $   \Lambda_\gamma: H^{1/2}(\partial B_R)  \to H^{-1/2}(\partial B_R)$ defined as
\begin{equation} \label{id:dnmap}
\Lambda_\gamma(f)(x) = \gamma  \partial_{\nu}   u(x)   \qquad x\in \partial B_R.
\end{equation}
Here $  \partial_\nu  $ denotes the normal derivative at the boundary, and  $u\in H^1(B_R)$ is the electrostatic potential given by the unique solution of the elliptic problem
\begin{equation}\label{id:gamma_BVP}
\left\{
\begin{array}{rl}
\nabla \cdot \left( \gamma \nabla u\right) = 0, & \mbox{ in  $B_R$}, \\
u=f, & \mbox{ on  $\partial B_R$}.
\end{array}
\right.
\end{equation}
The  Calderón inverse problem consists in recovering the conductivity $\gamma$ from the associated DtN (or voltage-to-current) map\footnote{Since $u\in H^1(B)$, a weak definition for $\Lambda_\gamma$ must be given in place of \eqref{id:dnmap}. An appropriate definition follows from the weak formulation of the boundary problem \eqref{id:gamma_BVP}.}. 
This, now classical, inverse problem can be seen as a simplified  mathematical model for {\it Electrical Impedance Tomography} (EIT), which  is related to important applications in medical imaging, nondestructive testing and geophysics. See for example the references \cite{Assenheimer_2001,cheney99,Isaacson_2006} for medical imaging applications. 

 From the mathematical point of view the first natural question is  {\it uniqueness}, which consists in determining if the map $\gamma \mapsto \Lambda_\gamma$ is injective. Uniqueness was originally proved for $C^2$ conductivities by Sylvester and Uhlmann in \cite{SU87} |see   \cite{CaroRogers16,Haberman15} for lower regularity uniqueness results| and a stability estimate was given by Alessandrini in \cite{Alessandrini88}.  
 
 Another important problem is  {\it reconstruction}, where the aim is to design a method to recover $\gamma$ from its associated DtN map.
 It is well known that the reconstruction problem is ill-posed due to the weak stability of the map $ \Lambda_\gamma \mapsto \gamma$. This means in particular that classical approaches based on least squares formulations are difficult to implement in practice. In dimension 3, a more successful attempt was addressed in the works \cite{BKM11,Delbary_2011,DHK12,DelbaryKnudsen14}, where the authors implemented a direct  reconstruction method, following the  scheme proposed by Nachman in \cite{nachman88}. However, the method is based on {\it Complex Geometrical Optic solutions} (CGOs) introduced in \cite{SU87} and, as a consequence, it requires to solve a complicated integral equation on the boundary involving highly oscillating functions; therefore it presents a number of difficult challenges at the numerical level.  We also refer to \cite{MuellerSiltanen20,MuellerSiltanenBook} for more references on the numerical aspects of EIT and for more information on the the two dimensional case that it is not addressed here. 
 
 For simplicity we assume for the moment  that $\gamma\in W^{2,\infty}(B_R,\R_+)$ and that both $\gamma-1$ and $\partial_{\nu} \gamma $ vanish  at $\partial B_R$. 
 Based on the works  \cite{SU87}  and   \cite{nachman88}, which give an explicit inverse formula for the map $\Lambda_\gamma \mapsto \gamma$, it is natural to consider   an approximation  to $\gamma$ that depends linearly on the DtN map.
 We will denote this approximation by $\gae$, which formally is given by the formula
\begin{equation}\label{eq:gexp}
    (\widehat{\gae-1})(\xi)=-\frac{2}{|\xi|^2}\lim_{|\zeta| \to \infty} \int_{\partial B_R} e^{-i x\cdot (\xi + \zeta)} (\Lambda_\gamma - \Lambda_0) e^{ix\cdot \zeta} , \quad \xi\in \mathbb{R}^d\setminus\{0\},
\end{equation}
where $\zeta \in \mathcal{V}_\xi =\{ \eta \in \mathbb{C}^d  \; :\; \eta \cdot \eta=0,\; (\xi+\eta)\cdot(\xi+\eta)=0\}$ (see \Cref{sec:2} for a detailed discussion). Here, $ \widehat{f} $ denotes the Fourier transform of a function $f$ defined as
\[
\widehat{f}(\xi) =\mathcal F_d (f)(\xi) = \int_{\R^d} e^{-ix\cdot \xi} f(x) \, dx.
\]
The expression \eqref{eq:gexp} has been considered in \cite{BKM11} and a similar definition has been used in \cite{KnudsenMueller11} as a way to obtain a numerical approximation to $\gamma$ that avoids solving a boundary integral equation in order to reconstruct the boundary values of CGOs. 
 The effectiveness of this   linearization method has been recently compared  to other  reconstruction methods    in \cite{HIKMTB21}  using synthetic data to simulate real discrete data from electrodes\footnote{It is worth mentioning that in \cite{BKM11,HIKMTB21}   the notation $\gae$ is used for a different approximation that is not linear on  $\Lambda_{\gamma}$, but is very similar in spirit to \eqref{eq:gexp}. The approximation \eqref{eq:gexp}  is denoted as $\gamma_{\mathrm{app}}$ in \cite{BKM11}.}.

The fact that $\gae$ is linear on the data, makes $\gae$ a close analogue of the Born approximation in scattering problems for Schrödinger operators $-\Delta +q$. 
For this reason from now on we will refer to $\gae$ as the {\em Born approximation} of $\gamma$. 
In scattering applications, the Born approximation   is widely used  as an acceptable approximation for the potential $q$.  
It is well known, for example, that it contains the leading singularities of $q$, and that  one can reconstruct  $q$ explicitly from its   Born approximation in the case of small potentials. See, respectively,  \cite{fix,back} and \cite{BCLV18,BCLV19} for recent results on both questions and for further references. This suggest that $\gae$ could have similar numerical applications.

Unfortunately, the definition of $\gae$ in \eqref{eq:gexp} is formal in two ways. On one hand, it is not clear if the limit in the right hand side exists. 
On the other hand,   even if the term in the right hand side were a well-defined function, there is no guarantee that it decays in $|\xi|$ fast enough in order to ensure   that it is  the Fourier transform of a function. 
These issues, together with the lack of a simple procedure to obtain $\gae$ from $\Lambda_\gamma$ |notice that \eqref{eq:gexp} involves oscillatory integrals and a high frequency limit|    has relegated the Born approximation to a marginal place in the numerical reconstruction aspect of the three-dimensional Calderón problem.


Recently,     explicit formulas for the Born approximation of a closely related  inverse problem |in which the elliptic operator in \eqref{id:gamma_BVP} is replaced by a Schrödinger operator $-\Delta +q$| have been obtained in \cite{born_aprox}. 
These formulas  can be adapted to the conductivity inverse problem very easily, and show that the limit \eqref{eq:gexp} exists  when the conductivity is a sufficiently smooth radial function:
\begin{equation}\label{id:radialc}
\gamma(x)=\gamma_0(|x|),\quad \gamma_0\in W^{2,\infty}_{\mathrm{even}}(\R,\R_+). 
\end{equation}
Moreover, they also show existence of the limit  in the non-radial case after an appropriate average. 

Considering radial conductivities provides an important simplification to the problem, since in this case the DtN map is a function of the spherical Laplacian on $\partial B_R$. 
As  a consequence, $\Lambda_\gamma$ is diagonal in the basis of spherical harmonics on $\partial B_R$. In fact, the eigenspaces of $\Lambda_\gamma$ are precisely the spaces of spherical harmonics of a fixed degree $k\in \N_0$.
We denote by $(\lambda_k^R[\gamma])_{k\in\N_0}$    the sequence of eigenvalues of the DtN map $\Lambda_\gamma$ in $\partial B_R$. One can show that (see for instance \cite{born_aprox}), when $\gamma=1$ one has $\lambda_k^R[1] = k/R $  and that the first eigenvalue always vanishes, that is $\lambda_0^R[\gamma] =0$.  From now on we will drop the super-index in the case $R=1$ and use the notation $\lambda_k[\gamma] = \lambda_k^1[\gamma]$.

A consequence of the results in \cite{born_aprox} (see   \Cref{sec:subsec_2_3} for more details) is that for a radial conductivity $\gamma$ as \eqref{id:radialc}, and such that  both $\gamma-1$ and $\partial_{\nu} \gamma $ vanish  at $\partial B_R$,  one has:
\begin{multline} \label{id:gae_formula}
\widehat{(\gae-1)}(\xi) = \\
    -\pi^{d/2}  \sum_{k=1}^\infty  \frac{ (-1)^k}{k! \Gamma(k+d/2)}
  \left(\frac{|\xi|}{2}\right)^{2k-2}  R^{2k+d-1}  \left( \lambda_{k}^R[\gamma] - \frac{k}{R} \right), \quad \xi\in\mathbb{R}^d\setminus 0 ,
\end{multline} 
with $d\ge 3$.

The series on the right hand side in \eqref{id:gae_formula} is absolutely convergent for all $\xi \in \R^d$ since $\lambda_k^R[\gamma] \lesssim k \|\gamma\|_{L^\infty}$, as follows from the   min-max principle applied to the eigenvalues of the DtN map.
However,  \eqref{id:gae_formula} is still formal since 
we do not have \textit{a priori} estimates that control the growth in the $\xi$ variable, necessary to define the inverse Fourier transform of \eqref{id:gae_formula} in a natural functional space. 

On the other hand, we notice that the right hand side of \eqref{id:gae_formula} is well-defined for radial conductivities $\gamma\in L^\infty(B_R,\R_+)$. As a consequence we will drop the regularity requirement \eqref{id:radialc} and assume just that  $ \gamma\in \cG^R$ where 
\[
\cG^R :=\{\gamma\in L^\infty(B_R,\R_+)\,:\, \exists c,\delta >0,\;\gamma\geq c,\,\gamma\in W^{2,\infty}(B_R\setminus B_{R-\delta})\}.
\]
The regularity condition close to the boundary is imposed beacuse the analysis requires that both $\gamma-1$ and $\partial_{\nu} \gamma $ are well defined and vanish  at $\partial B_R$.
 
In fact,   in the  numerical experiments in \Cref{sec:3_gamma}     the right hand side of \eqref{id:gae_formula}
   decays quickly for large $|\xi|$, even for conductivities with jump discontinuities at the interior.
 This is  an interesting fact since in principle the heuristic derivation of \eqref{eq:gexp} requires conductivities with two bounded derivatives \eqref{id:radialc}, and leads us to conjecture that $ \gae$ is indeed always well defined    for conductivities in $\cG^R$.
Notice also that this class is strictly larger than the one for which the uniqueness of the inverse problem is proved. 

We have tested formula \eqref{id:gae_formula} in different situations with high precision values of $\lambda_k^R[\gamma]$ to allow us to explore which features of $\gamma$ can be recovered by $\gae$. All the numerical experiments involve radial conductivities in $B_R\subset \R^3$. In particular we find:

 \begin{enumerate}[label={\Roman*}),wide] 
\item \label{item:ga_1_int} {\bf Local uniqueness from the boundary.}  Assume that $\gamma_1$ and $\gamma_2$ are two different radial continuous conductivities. Then  $\gamma_1(x) = \gamma_2(x)  $   for $|x|>b$ implies that also $\gexp{\gamma_1}(x)=\gexp{\gamma_2}(x) $ for $|x|>b$. 

\item  \label{item:ga_2_int} {\bf Approximation close to the boundary} $\gae$ is in general a better approximation for $\gamma$ in a region close to the boundary of the ball than close to the center of the ball. This is more or less independent of the size of $\gamma$, and only breaks down for very large conductivities.     The fact that $\gae(x)$ is a better approximation of $\gamma(x)$ when $x$ is close to the boundary  has already been observed  in \cite{BKM11}.
 
\item  \label{item:ga_3_int} {\bf Approximation of small conductivities.} As expected, $\gae$ is an excellent approximation for $\gamma$ if $\norm{\gamma-1}_{L^\infty(\R^3)} $ is small, and the quality of the approximations  provided by $\gae$ worsens as this difference  becomes larger.

 \end{enumerate}

While property \ref{item:ga_3_int} is to be expected, properties \ref{item:ga_1_int} and \ref{item:ga_2_int} are truly remarkable, and suggest that $\gae$ contains important information on the conductivity, even when the conductivities are not necessarily very close to $\gamma =1$. The numerical experiments made to study property \ref{item:ga_1_int} have been motivated by  \cite[Theorem 1.4]{daude2} in which a  uniqueness result that has an analogous locality behaviour from the boundary is proved.
As far as we know these properties of the Born approximation do not have similar counterparts in inverse scattering problems.  The particular case $\gamma_2 = 1$ in property \ref{item:ga_1_int} is rigorously shown to hold under the assumption that \eqref{id:radial_born} is a tempered distribution (see   \Cref{sec:appendix_support}). Notice that this case implies the following property:
 \begin{enumerate}[label={\Roman*}),wide] \setcounter{enumi}{3}
\item \label{item:ga_4_int} {\bf Support from the boundary.} If $\gamma(x) -1$ vanishes for $|x|>b$ then $\gae(x)-1$ vanishes for $|x|>b$.
 \end{enumerate}

Another very natural question in this context is   to understand in which sense the Born approximation could recover the leading singularities of $\gamma$. Recovery of singularities results from the Born approximation are well known in Inverse Scattering, as mentioned previously. 
In the two-dimensional Calderón problem a detailed scheme to reconstruct the singularities from bounded conductivities has been given in   \cite{GLSSU2018}, where is also noted that this question is relevant for medical applications of EIT (classifying strokes as ischemic or hemorrhagic, for instance, see references in  \cite{GLSSU2018}).  
In this work  we are not able to capture numerical evidence of recovery of singularities  for $\gae$, though we get strong numerical evidence of this phenomenon when considering the   inverse boundary problem for the Schrödinger operator $-\Delta+q$. The numerical experiments presented in \Cref{sec:3_potential} show clearly that the Born approximation for this problem contains the same jump discontinuities of the original potential, even for very large potentials.


The Born approximation $\gae$ and formula \eqref{id:gae_formula} are useful to gain new insights on the reconstruction aspect of the Calderón problem. For notational simplicity let us illustrate this in the particular case $R=1$. First, note that the Born approximation contains  the same information as the DtN map since one can show that  
\begin{equation} \label{id:lambda_as_moments}
\lambda_{k}[\gamma] -k =  2k(k+d/2-1) \frac{1}{|\IS^{d-1}|}\int_{B} (\gae(x)-1) |x|^{2(k-1)} \, dx  ,
\end{equation}
under the assumption that $\gae$ is well defined\footnote{This is true under the assumption that $\gae \in L^1(B)$ . To prove rigorously a weak version of \eqref{id:lambda_as_moments} one has to show that $|x|^{2k} \gae(x)$ is a tempered distribution for all $k>k_0$ for some $k_0\in \N_0$.}  (see \Cref{sec:subsec_moments}).
This implies that formally   $\gae$ determines $\gamma$, since $\Lambda_\gamma$ determines the conductivity (at least  in the appropriate function spaces where uniqueness holds). 
Thus, from the point of view of reconstruction, properties \ref{item:ga_1_int}--\ref{item:ga_3_int}  imply that $\gae$ is an equivalent, but more useful, way to present the measurement data than the DtN map.




On the other hand, the Born approximation does not provide a way around the bad stability characteristic of the Calderón problem. 
In fact, the identity \eqref{id:lambda_as_moments} implies that finding the Born approximation is essentially equivalent to solving a Hausdorff moment problem, which is also an ill-posed problem with logarithmic stability |see \cite{daude2} 
for more details on the relation between the moment problem and the radial Calderón problem. This instability is reflected in   the numerical experiments: the implementation of  formula \eqref{id:gae_formula}  is not difficult but it  requires very precise values of $\lambda_k[\gamma]-k$.
In this work, the precision that we have used in the computation of the values of $\lambda_k[\gamma]-k$ goes well beyond what one can expect in any real application, since one of the objectives is to gain insights on the general behaviour of $\gae$.

The previous  observations open an intriguing possibility. 
If one can obtain $\gae$ from the DtN map, and $\gamma$ from $\gae$, the reconstruction problem is factorized in a linear part and a non-linear part. 
The linear part is the Hausdorff moment problem that is formally solved by \eqref{id:gae_formula}. 
This raises the question of understanding whether or not the actual stability estimate of the non-linear part is better than logarithmic. An affirmative answer to this question would imply that the bad stability properties of the radial Calderón problem are caused exclusively by the fact that one is implicitly solving a moment problem, and that the Born approximation has already made the hard work of decompressing the information contained in the DtN map.
 In this sense,    numerical observations \ref{item:ga_1_int}--\ref{item:ga_3_int} about $\gae$ seem to give some  evidence, albeit indirect and very limited, that this could be true. This will be discussed with more detail in \Cref{sec:subsec_moments}.

Another interesting issue that is analyzed here is the dependence on $R$ of the quality of the Born approximation of a fixed conductivity that is defined on $\R^d$. 
Assume that $\gamma-1$ is supported in $B$. 
In principle \eqref{id:gae_formula} yields a different Born approximation $\gae(x;R)$ for each value of $R\ge 1$. 
It is possible to find an identity relating the eigenvalues of $(\lambda_k[\gamma])_{k\in\N_0}$ of the DtN map $\Lambda_\gamma$ in $\partial B$ to the eigenvalues of the DtN map in $\partial B_R$ (see \Cref{sec:subsec_scattering}). This yields the following formula
\begin{multline}  \label{id:gae_scattering_0}
\mathcal F_d{(\gae(\centerdot,R)-1)}(\xi) =  
    -\pi^{d/2}  \sum_{k=1}^\infty  \frac{ (-1)^k}{k! \Gamma(k+d/2)}
  \left(\frac{|\xi|}{2}\right)^{2k-2}     \left( \lambda_{k}[\gamma] - k  \right) \times .... \\
  \times  \frac{2k+d-2}{\lambda_k[\gamma] + k +d-2 - R^{-(2k+d-2)}(\lambda_k[\gamma] - k)}  ,
\end{multline} 
which taking the limit $R \to \infty$ becomes
\begin{multline}  \label{id:gae_scattering}
\mathcal F_d{(\gae(\centerdot,\infty)-1)}(\xi) =    \\
    -\pi^{d/2}  \sum_{k=1}^\infty  \frac{ (-1)^k}{k! \Gamma(k+d/2)}
  \left(\frac{|\xi|}{2}\right)^{2k-2}      \left( \lambda_{k}[\gamma] - k \right)   
    \frac{2k+d-2}{\lambda_k[\gamma] + k +d-2}  .
\end{multline} 
This identity shows that one can define a scattering limit $\gae(x;\infty)$ for the Born approximation.  The numerical experiments in \Cref{sec:numerico} suggest that in general ${\gae}(x;R_2)$ is a worse approximation for $\gamma$ than  ${\gae}(x;R_1)$ if $R_2>R_1$, but the deterioration stabilizes very quickly as $R$ grows, as suggested by formula \eqref{id:gae_scattering}. See  \Cref{sec:subsec_scattering} for more consequences of these formulas in the context of the inverse boundary problem for the Schrödinger operator $-\Delta +q$.

Nonetheless, the Born approximation can be used as the basis of a simple but useful algorithm to reconstruct conductivities close to $\gamma=1$ from the DtN map. Since $\gae$ is obtained after linearizing the Calderón problem around $\gamma = 1$, one expects the difference between $\gamma$ and $\gae$ to be of  quadratic order in the norm of $\gamma-1$. This suggests that the following fixed point algorithm can be used to improve the Born approximation when the conductivities are close enough to $\gamma =1$:
\begin{equation} \label{eq:ga_itera_1}
\begin{cases}
\begin{alignedat}{2}
    &\gamma^0 & &=  \; \gae ,  \\
     &\gamma^{n+1} & &= \;  \gae + \gamma^{n} -[\gamma^{n}]_{\mathrm{exp}} , \quad n\geq 0.
     \end{alignedat}
\end{cases}
\end{equation}
We provide numerical evidence on the fast convergence for continuous conductivities that are not far from $1$ (see \Cref{label:algorithm}). 
Similar iterative algorithms have been used in scattering theory, see \cite{BCR16,BCLV18}. 
In the context of the Calderón problem, see \cite{GH22arxiv} for a recent reconstruction algorithm for small conductivities in dimension 2 in which convergence and stability  are   proved.
Note that, in order to compute the new iteration $\gamma^{n+1}$ in \eqref{eq:ga_itera_1}, we have to approximate the Born approximation of $\gamma^n$ and this requires its DtN map. 
Therefore, the efficiency of the algorithm strongly depends on the existence of an accurate and fast algorithm to solve the direct problem, i.e. to compute the eigenvalues of $\Lambda_\gamma$ from $\gamma$. Fortunately, in the case of radial  conductivities such algorithm is available (see \cite{BKM11}).

The rest of the article is divided as follows. In \Cref{sec:2}   we will introduce the Born approximation for the inverse boundary problem for the Schrödinger $-\Delta +q$  and we will analyze the scattering limit for this problem. We then show how to derive \eqref{id:gae_formula} and  study the relation of the Born approximation with the  inverse Hausdorff   moment problem. In \Cref{sec:numerico}  we describe some implementation details to obtain the Born approximation and include   the numerical experiments    for the potential and for the conductivity problems. Finally in  \Cref{label:algorithm} we   provide numerical evidence on the   convergence of the numerical algorithm \eqref{eq:ga_itera_1} and an analogous version of the algorithm for the  Schrödinger operator $-\Delta +q$. The \Cref{sec:appendix_support} provides a short proof of property \ref{item:ga_4_int} under the assumption that the Born approximation is  well defined as a tempered distribution.

\subsection*{Acknowledgments} 
This research has been supported by Grant \mbox{MTM2017-85934-C3-3-P} of Agencia Estatal de Investigación (Spain). The authors would like to thank Thierry Daudé and François Nicoleau for very insightful discussions on the radial Calderon problem.


\section{Two closely related inverse  problems} \label{sec:2}

In this section we introduce the Born approximation for the inverse boundary problem associated to the Schrödinger operator $-\Delta +q$ in  \Cref{sec:subsec_2_1}, and we analyze the scattering limit for this problem in \Cref{sec:subsec_scattering}. We then show how to derive \eqref{id:gae_formula}     in \Cref{sec:subsec_2_3}. Finally in  \Cref{sec:subsec_moments} we study the relation of the Born approximation with the  inverse Hausdorff   moment problem, and its implications for the numerical reconstruction in the Calderón problem.

\subsection{The Born approximation in the inverse Schrödinger  boundary value problem} \label{sec:subsec_2_1}

In \cite{SU87}, Sylvester and Uhlmann showed in particular that $\Lambda_\gamma$ determines $\gamma$ for $C^2(B_R)$ conductivities in $d\ge 3$ by reducing the problem to an inverse boundary value problem for the Schrödinger operator $-\Delta +q$. We next recall the main ingredients of their strategy. For $\gamma\in W^{2,\infty}(B_R,\R_+)$ with $\gamma =1$ at the boundary, let 
\begin{equation}\label{id:qgamma}
q=\frac{\Delta \sqrt{\gamma}}{\sqrt{\gamma}} \quad \text{and} \quad v=\sqrt{\gamma}u.
\end{equation}
Then
 \begin{equation}  \label{id:calderon_q} 
\left\{
\begin{array}{rll}
-\Delta v + qv  &= 0  &\text{in }  B_R,\\
v   &= f   &\text{on }   \partial B_R.  \\
\end{array}\right.
\end{equation}
if and only if $u$ satisfies \eqref{id:gamma_BVP}.

Consider the family 
\begin{equation} \label{id:Qdef}
\cQ_d^R:=\{q\in L^\infty(B_R,\R) \,:\, 0\not\in \Spec_{H^1_0(B_R)}(-\Delta+q)\}. 
\end{equation}
Note that $\cQ_d^R$ contains all the $q\in L^\infty(B_R,\R)$ that are obtained from some $\gamma\in W^{2,\infty}(B_R,\R_+)$ via \eqref{id:qgamma}. 
Again, the DtN map  $\Lambda[q]$ associates to a function $f\in H^{1/2}(\partial B_R)$ the normal derivative $\Lambda[q] (f):=\partial_\nu v|_{\partial B_R} \in   H^{-1/2}(\partial B_R)$, where $v$ is  the unique solution $v\in H^1(B_R)$ of \eqref{id:calderon_q}.
In particular, for $q$ as in \eqref{id:qgamma} one has that 
\begin{equation} \label{id:dn_maps}
\Lambda  [q] f = \gamma^{-1/2}  \left( \Lambda_\gamma  +\frac{1}{2}    \partial_\nu\gamma \right) \gamma^{-1/2} f, \quad \text{on } \, \partial B_R ,
\end{equation}
and hence $\Lambda[q] = \Lambda_\gamma$ whenever $(\gamma-1)\in W^{2,\infty}_0(B_R)$. 

These observations show that the uniqueness and reconstruction questions for $\Lambda_\gamma$ are reduced to the unique determination and reconstruction of a potential $q\in L^\infty(B_R,\R)$ from the Dirichlet to Neumann map  $\Lambda[q]$ associated to the Schrödinger operator $-\Delta+q$ on $B_R$.

Assume now that $q(x) = q_0(|x|)$ is a radial function such that $q\in\cQ_d^R$, not necessarily arising from a conductivity via \eqref{id:qgamma}. For such $q$, the DtN map is diagonal in the spherical harmonics.
From now on we denote the eigenvalues of $\Lambda[q]$ at $\partial B_R$ by  $(\lambda_k^R[q])_{k\in \N_0}$,  and $(\lambda_k[q])_{k\in \N_0}$ in the case $R=1$.

In this inverse problem the Fourier transform of the Born approximation can be defined as the limit
\begin{equation}\label{id:q_exp_def}
    \widehat{\qe}(\xi) = \lim_{|\zeta| \to \infty} \int_{\partial B_R} e^{-i x\cdot (\xi + \zeta)} (\Lambda[q] - \Lambda_0) e^{ix\cdot \zeta} , \quad \xi\in \mathbb{R}^d,
\end{equation}
where $\zeta \in\mathcal V_\xi$, as in the case of \eqref{eq:gexp}, see \cite{born_aprox} for more details and references  on the origin and motivation of this definition.

Recently, in \cite[Theorem 1]{born_aprox} the authors introduced an explicit formula for the Born approximation $\qe$ in the radial case when the domain is the unit ball $B$:
\begin{equation}  \label{id:radial_born}   
    \widehat{\qe}(\xi)
    = 2 \pi^{d/2}  \sum_{k=0}^\infty  \frac{ (-1)^k}{k! \Gamma(k+d/2)}
    \left(\frac{|\xi|}{2}\right)^{2k}   (\lambda_k[q] - k) .
\end{equation}
An analogous but more complex formula is also  proved for the non-radial case, see \cite[Theorem 3]{born_aprox}.

As in the case of the conductivity, we remark that \eqref{id:radial_born}   is  formal.  If $q$ is supported in $B_\alpha \subseteq B$, it   follows from \cite[Theorem 2]{born_aprox} that  
\begin{equation} \label{est:moment_aprox_radial}
    | \lambda_k[q] -k|  \lesssim   \|q\|_{L^\infty(B)}\frac{\alpha^{2k}}{2k+d},
\end{equation}
for $k > \alpha \norm{q}_{L^\infty(B)}^{1/2} - \frac{d-2}{2}$ . 
This estimate implies that  the series in \eqref{id:radial_born} is absolutely convergent for all $\xi \in \R^d$, but there is not an a priori control of the growth in the $\xi$ variable. This is a subtle problem that will be discussed again in \Cref{sec:subsec_moments}.

Formula \eqref{id:radial_born} provides an easy way to approximate the Born approximation numerically. As expected, the numerical results indicate that it shares many properties with the conductivity case:
  \begin{enumerate}[label={\Roman*.b}),wide]
  
    \item  \label{item:q_1} For bounded potentials, the Born approximation $\qe$ is well defined as an inverse Fourier transform of \eqref{id:radial_born} since $\widehat{q}(\xi)$ decays to 0 as $|\xi| \to \infty$,
    provided that the potential $q= q_+ +q_-$ has a  not very large negative part $q_-$. It is not clear if the Born approximation   is well defined for very large and negative potentials, since in the numerical experiments \eqref{id:radial_born} becomes very large and separates from $\widehat{q}(\xi)$ as $|\xi| \to \infty$. 
    
     \item  \label{item:q_3} $\qe$ is in general a better approximation for $q$ in a region close to the boundary of the ball than close to the center of the ball.  
 
    \item   \label{item:q_2} As in the case of the conductivity, if $q_1$ and $q_2$ are two different potentials and $q_1(x) = q_2(x)  $   for $|x|>b$, it follows  that $\gexp{q_1}(x)=\gexp{q_2}(x) $ for $|x|>b$.

    \item  \label{item:q_4} The Born approximation recovers the jump  discontinuities of $q$. 
      This is known as recovery of singularities and it is well established in other scattering problems, as mentioned in the introduction.
 
    \item  \label{item:q_5} In most of the numerical examples, given a potential $q$ and a parameter $t\in (0,\infty)$, the Born approximation of $tq$ develops a characteristic oscillatory behaviour close to $x=0$ as $t$ grows. 

\end{enumerate}

In analogy with the iterative algorithm \eqref{eq:ga_itera_1}, we propose the following algorithm  to improve the Born approximation for the Schrödinger inverse boundary problem: 
\begin{equation} \label{eq:itera_q}
\begin{cases}
\begin{alignedat}{2}
    &q^0 & &=  \;\qexp{q} ,  \\
     &q^{n+1} & &= \;  \qexp{q} + q^{n} -\qexp{q^{n}}, \quad n\geq 0 . 
     \end{alignedat}
\end{cases}
\end{equation}
As we  show in the numerical experiments below, this algorithm approximates very fast potentials that are not large in the $L^\infty$ norm.

\subsection{The scattering limit} \label{sec:subsec_scattering} \

It is possible to extend the Born approximation \eqref{id:radial_born} to a formula that is valid for any  ball of radius $R>0$.
Let $q\in \bigcap_{R\ge 1} \cQ^R_d$. We denote momentarily  the DtN map of $q$ in  $\partial B_R$ by $\Lambda^R[q]$.  For convenience,
 we will omit the $R=1$ superscript when $R=1$   as in the case of the eigenvalues.
Also, we will denote $\qe$ in $B_R$ as $\qe(x;R)$ and its Fourier transform by $\widehat{\qe}(\xi;R)$.  
 The following identity for the Born approximation of a potential in $B_R$ follows from a straightforward change of variables in formula \eqref{id:radial_born}:
\begin{equation}  \label{id:qexp_formula_R_1}
\widehat{\qe}(\xi;R)
 =  2 \pi^{d/2}  \sum_{k=0}^\infty  \frac{ (-1)^k}{k! \Gamma(k+d/2)}
  \left(\frac{|\xi|}{2}\right)^{2k}  R^{2k} R^{d-1} \left (\lambda_k^R[q] - \frac{k}{R} \right ) .
\end{equation} 

 In fact, the change of variables $y=x/R$, $q_R(y)= R^2 q(Ry)$ transform the initial problem \eqref{id:calderon_q} into the corresponding one for the unit ball $B = B_1$ with potential $q_R$. Also, with this change of variables \eqref{id:q_exp_def} implies that $\qe(Ry;R) =[q_R]_{\mathrm{\mathrm{exp}}}(y;1)$. This proves \eqref{id:qexp_formula_R_1}.
Formula \eqref{id:gae_formula} for the conductivity   follows easily by combining the formula \eqref{id:qexp_formula_R_1} with the linearization of \eqref{id:qgamma}, as we will   show in \Cref{sec:subsec_2_3}. 

 An interesting question is how the Born approximation of a fixed potential might depend on $R$. That is, we want to compare the different Born approximations $ {\qe}(x;R)$ for the same potential $q$.
We start with the following simple lemma.
\begin{lemma} \label{lemma:eigenvalues}
Let $d\ge 2$, $R\geq 1$, and 
 $q = q_0(|\cdot|)$, where $q_0 \in L^\infty(\R_+,\R)$ is supported in $(0,1]$. Assume that $q \in \cQ_d^R$ for all $R\ge 1$. Then the eigenvalues $(\lambda_k[q])_{k\in\N_0}$ of $\Lambda[q]$ in the unit ball determine the eigenvalues $(\lambda_k^R[q])_{k\in\N_0}$ for all $R\ge 1$:
\begin{equation} \label{id:R_eigenvalues}
\lambda_k^R[q] - \frac{k}{R} = R^{-(2k+d-1)} (\lambda_k[q] - k) \frac{2k+d-2}{\lambda_k[q] + k +d-2 - R^{-(2k+d-2)}(\lambda_k[q] - k)} .
\end{equation}
\end{lemma}
\begin{proof}
  Let $Y_k \in C^\infty( \IS^{d-1})$ be a spherical harmonic of order $k$. Let $v_k$ be the solution of \eqref{id:calderon_q} with $f(x) = Y_k(x/|x|)$. Since the potential is radial, by separation of variables one has that $v_k(x)  = Y_k(x/|x|) u_k(|x|)$, where $u_k$ is a solution of
 \begin{equation}   \label{id:radial_b_problem}
\begin{cases}
\displaystyle{- \frac{1}{r^{d-1}} \frac{d\phantom{r}}{dr} \left( r^{d-1} \frac{d \phantom{r}}{dr}u_k \right) + \left(q_0(r) + \frac{1}{r^2}k(k+d-2)\right)u_k(r)}  = 0  \quad \text{in } (0,R),\\ \vspace{1mm}
u_k(0) < \infty,\quad u_k(R) =1   .     
\end{cases}
\end{equation} 
Note that for every $1\leq R'\leq R$ one has
\[
\Lambda^{R'}[q] (Y_k) = \lambda_k^{R'}[q] Y_k =\frac{u_k'(R')}{u_k(R')} Y_k\,\implies\, \lambda_k^{R'}[q] = \frac{u_k'(R')}{u_k(R')}.
\]
Therefore, if $R\geq 1$, since $q_0$ is supported in $(0,1]$ we have that $u_k(r)$ is a solution of the free equation in $(1,R)$. It follows that
\[
u_k(r) = Ar^{k} + Br^{-(k+d-2)}, \quad \text{in } (1,R) \; \text{ with } \; AR^k +BR^{-(k+d-2)} =1 .
\]
Since the restriction of $u_k$ to $[0,1]$ is also a solution of \eqref{id:radial_b_problem} with $R=1$, by the previous discussion we have that
\[
\lambda_k[q] = \frac{u_k'(1)}{u_k(1)} =  \frac{Ak - B(k+d-2)}{A + B}.
\]
Hence we have the linear system
\[
\begin{cases}
Ak - B(k+d-2) = \lambda_k [q](A+B), \\ 
 AR^k +BR^{-(k+d-2)} =1  ,
\end{cases}
\]
which can  easily be solved to find $A$ and $B$ in terms of $\lambda_k[q]$, $k$, and $R$. Then, plugging the solutions into the identity
\[\lambda_k^R[q] = u_k'(R)= kAR^{k-1} -(k+d-2) BR^{-(k+d-1)} ,\]
yields formula \eqref{id:radial_b_problem} after some computations. 
\end{proof}

As a consequence of this lemma, if $d \ge 3$, we can write \eqref{id:qexp_formula_R_1} in terms of the eigenvalues  $(\lambda_k[q])_{k\in\N_0}$ of $\Lambda[q]$ in the unit ball:
 \begin{multline} \label{id:qexp_formula_R_2}
\widehat{\qe}(\xi;R) =  \\
   2 \pi^{d/2}  \sum_{k=0}^\infty  \frac{ (-1)^k}{k! \Gamma(k+d/2)}
  \left(\frac{|\xi|}{2}\right)^{2k}   \left (\lambda_k[q] - k \right ) \frac{2k+d-2}{\lambda_k[q] + k +d-2 - R^{-(2k+d-2)}(\lambda_k[q] - k)}  .
\end{multline} 
Therefore the  limit  $R\to \infty$ of the previous expression is well defined:
\begin{equation} \label{id:qexp_formula_R_3}
\widehat{\qe}(\xi;\infty)
 =  2 \pi^{d/2}  \sum_{k=0}^\infty  \frac{ (-1)^k}{k! \Gamma(k+d/2)}
  \left(\frac{|\xi|}{2}\right)^{2k}   \left (\lambda_k[q] - k \right )\frac{2k+d-2}{\lambda_k[q] + k +d-2} .
\end{equation} 
This can be understood as a scattering limit for the Born approximation.  Formulas \eqref{id:gae_scattering_0}  and \eqref{id:gae_scattering} follow easily from the previous identities, as we will show in \Cref{sec:subsec_2_3}.

It is well known that the inverse boundary value problem for the Schrödinger operator is closely related to the fixed energy scattering problem as $R\to \infty$, see \cite{Uh92}. Therefore, the fact that $\qe(x;\infty) $ is well defined   opens the possibility of obtaining a Born approximation in terms of the fixed angle scattering data related to $\qe(x;\infty)$ and formula \eqref{id:qexp_formula_R_1} as $R \to \infty$.

Notice that the fact that $\widehat{\qe}(\xi;R)$  is not linear on $\lambda_k[q]$  is a consequence of the non-linear identity \eqref{id:R_eigenvalues} relating  $\lambda_k^R[q]$ and  $\lambda_k[q]$. After linearizing \eqref{id:qexp_formula_R_2} and \eqref{id:qexp_formula_R_3} with respect the eigenvalues $\lambda_k[q]$, one recovers exactly \eqref{id:radial_born}. 
As in the case of the conductivity, for a fixed potential  $q\in \bigcap_{R\ge 1} \cQ^R_d$  we will later show numerically that $\qe(x;R)$ becomes a worse approximation for $q$ as $R$ grows.
This suggest that it is always a better strategy to obtain $\lambda_k[q]$ from $\lambda_k^R[q]$ with \Cref{lemma:eigenvalues} and use  \eqref{id:radial_born} rather than using directly \eqref{id:qexp_formula_R_1}.


\subsection{The formula for \texorpdfstring{$\gae$}{ \textgamma exp}} \label{sec:subsec_2_3} \

Formula \eqref{id:gae_formula} follows from  linearizing  \eqref{id:calderon_q}.
First, we assume that $\gamma-1\in W^{2,\infty}_0(B_R,\R_+)$. By \eqref{id:dn_maps} this implies that   
\begin{equation} \label{id:ga_q_dnmap}
\Lambda_\gamma=\Lambda [q] ,
\end{equation}
where $q = \gamma^{-1/2}\Delta \gamma^{1/2}$. The map  $\gamma \mapsto q$ is a non-linear map  between $W^{2,\infty}(B_R)$ and $L^\infty(B_R)$.
It is not difficult to verify that the Fréchet derivative of this map at $\gamma =1$, is given by the operator $\frac12 \Delta $. Thus
\[
q =  \frac12 \Delta (\gamma -1) + o(\norm{\gamma-1}_{W^{2,\infty}(B_R)}).
\]
Since $\gamma-1$ and $q$ have compact support, we can extend both functions to $\R^d$ by zero and take the Fourier transform of the previous expression. This yields
\[
\widehat{(\gamma-1)}(\xi) = -2 \frac{\widehat{q} (\xi)}{|\xi|^2} + R(\xi),
\]
where $R$ is some remainder term.
Combining the linearization $\qe$ of the Schrödinger problem with the linearization 
\[
\widehat{(\gamma-1)}(\xi) \sim -2 \frac{\widehat{q} (\xi)}{|\xi|^2} ,
\]
gives   the following definition for $\gae$:
\begin{equation} \label{id:exp_rel}
 \widehat{(\gae-1)}(\xi) := -2 \frac{\widehat{\qe} (\xi)}{|\xi|^2} ,
\end{equation}
which is how one obtains   \eqref{eq:gexp}  from  \eqref{id:q_exp_def}.

  This is essentially the same argument given in \cite{BKM11} to motivate the definition of $\gae$. 
Notice also that now  \eqref{id:gae_formula} follows directly  from    and \eqref{id:qexp_formula_R_1}. In the same way, formulas \eqref{id:gae_scattering_0}  and \eqref{id:gae_scattering}  also follow immediately from  \eqref{id:qexp_formula_R_2} and \eqref{id:qexp_formula_R_3}.

Since $\gae$ is obtained from $\qe$, in principle one requires the conductivity to have two bounded derivatives, but nothing prevents us to use \eqref{id:gae_formula} for less regular conductivities as mentioned in the introduction. In fact, one important advantage of \eqref{id:gae_formula} is that it does not require to go through the Schrödinger problem: it only involves the DtN map $\Lambda_\gamma$.

 \subsection{Connection with the Hausdorff moment problem}   \label{sec:subsec_moments} \
 
 In this section we show that formula \eqref{id:radial_born}  implies that the linearization of the radial  Calderón problem is essentially a Hausdorff moment problem, a connection observed in \cite{daude2} and \cite{born_aprox}. To see this, let $f(x) = f_0(|x|)$ be any radial $L^1(\R^d)$ function with compact support. Define the moments
\begin{equation} \label{id:moments}
  \sigma_k[f] = \frac{1}{|\SS^{d-1}|}\int_{\R^d} f(x) |x|^{2k} \, dx = \int_0^\infty f_0(r) r^{2k+d-1} \, dr.  
\end{equation}
Then, it holds that
\begin{equation} \label{id:series}
\widehat{f}(\xi)
 =  2 \pi^{d/2}  \sum_{k=0}^\infty    \frac{ (-1)^k}{k! \Gamma(k+d/2)}
  \left(\frac{|\xi|}{2}\right)^{2k}  \sigma_k[f] ,
\end{equation}
where the series on the right hand side converges absolutely for radial compactly supported $f \in L^1(\R^d)$ |see \cite[Section 3]{born_aprox} for a short proof of this fact\footnote{Notice that the compact support of $f$ is essential. It is not difficult to find counterexamples to this formula in the Schwartz class.}. 
This means that formula \eqref {id:series} is a formal solution for the following Hausdorff  moment problem: given a sequence of numbers $( \mu_k)_{k\in \N_0}$, find a compactly supported and radial function $f$ such that \eqref{id:moments} holds with $\sigma_k[f] =\mu_k$. 

The similarity between formulas \eqref{id:radial_born} and \eqref {id:series}  implies that the linearization of the radial Calderón problem is equivalent to the previous Hausdorff moment   problem. 
 In fact, formally one can think of $\qe$ as the radial and compactly supported function such that $\sigma_k[\qe] = \lambda_k[q] -k$, provided that such a function exists.

On the other hand, if $q$ is supported in $B_\alpha \subseteq B$, it   follows from \cite[Theorem 2]{born_aprox} that  
\begin{equation} \label{est:moment_aprox_radial_2}
    \left| \lambda_k[q] -k  - \sigma_k[q]  \right |\lesssim   \|q\|^2_{L^\infty(B)}\frac{\alpha^{2k}}{(2k+d-2)^3},
\end{equation}
for $k > \alpha \norm{q}_{L^\infty(B)}^{1/2} - \frac{d-2}{2}$ . 
  An important consequence of this estimate is that the linear map $q\mapsto\sigma_k[q]$ is the Fréchet derivative of  the non-linear functional $q\mapsto\lambda_k[q]$ on  $L^\infty(B)$ at $q=0$.  Thus, estimate \eqref{est:moment_aprox_radial_2} offers a different way to understand the connection between the radial Calderón problem and the Hausdorff moment problem.
  This has also been observed in a more general setting in \cite{daude2}.

The same connection can be established in the case of the conductivity problem, as we saw in the introduction. To see this, we now justify  identity \eqref{id:lambda_as_moments}.
First, it is convenient to introduce the following notation. Let $\mathcal L_d: \ell^{\infty}(\N_0) \times \R^d \to C^\infty(\R^d)$  be the following operator:
\begin{equation*}
   \mathcal L_d \left((\mu_k)_{k\in \N_0} ;\xi\right)   
 = 2 \pi^{d/2}  \sum_{k=0}^\infty  \frac{ (-1)^k}{k! \Gamma(k+d/2)}
  \left(\frac{|\xi|}{2}\right)^{2k}  \mu_k ,
\end{equation*}
where $(\mu_k)_{k\in \N_0}$ is   a sequence of reals numbers belonging to $\ell^\infty(\N_0)$. With this notation we can respectively write \eqref{id:radial_born}   and \eqref{id:series} as  
\begin{equation*}
    \widehat{\qe}(\xi) =  \mathcal L_d \left((\lambda_k[q]-k)_{k\in \N_0} ;\xi\right) \quad \text{and} \quad  \widehat{f}(\xi) =  \mathcal L_d \left((\sigma_k[f])_{k\in \N_0} ;\xi\right).
\end{equation*}
 The first identity together with \eqref{id:ga_q_dnmap} and \eqref{id:exp_rel} imply that
\begin{multline} \label{id:alternative_gaexp}
   \widehat{(1-\gae)}(\xi) = -2\frac{ \mathcal L_d \left((\lambda_k[\gamma]-k)_{k\in \N_0} ;\xi\right) }{|\xi|^2} \\
 =\mathcal L_d \left( \left (\frac{1}{2(k+1)(k+d/2)} (\lambda_{k+1}[\gamma]-k) \right)_{k\in \N_0} ;\xi\right)   ,
\end{multline}
 which is a completely equivalent way to write \eqref{id:gae_formula} in the case $R=1$. From the previous expression identity \eqref{id:lambda_as_moments} follows directly, provided that $\gae$ belongs to an appropriate functional space. 

We have already mentioned that the definition of $\qe$ and $\gae$ is formal, since we do not control the growth of the right hand sides   of \eqref{id:gae_formula} and \eqref{id:radial_born}. This question now becomes equivalent to the non-trivial matter of determining if $(\lambda_k[q]-k)_{k\in\N_0}$  is the moment sequence corresponding to a    function (or  distribution), and analogously for the conductivity. As we have already mentioned, the numerical experiments in this work suggest that this is the case for conductivities in $\mathcal G^R$, and for potentials satisfying certain conditions |see property \ref{item:q_1}.
 

 As mentioned previously, the problem of recovering a conductivity from the DtN map in the Calderón problem presents a number of numerical challenges. The Born approximations $\qe$ and $\gae$ offer a simpler setting in which study these problems, by giving an approximation to $q$ and $\gamma$, respectively, that can be computed easily from the spectrum of the DtN map, and which contains useful information on the conductivity or the potential. Nonetheless, the stability challenges that appear in the reconstruction problem  also appear to compute  the Born approximation as can be expected from the connection with the Hausdorff moment problem.
 
The Hausdorff moment   problem is a notoriously ill-posed inverse problem, see for example \cite{moment_theory_book}. In fact, it   is related to the inversion of the Laplace transform, which is also an ill-posed problem (to see this, use the change of variables $r=e^{-t}$ in the last integral of \eqref{id:moments}). Analogously to the Calderón problem, the   Laplace transform and the forward map in the Hausdorff moment problem are injective under suitable conditions,   but the inverse mappings are not continuous. This affects stability, which can only be achieved in compact subspaces of potentials or conductivities. These are called conditional stability estimates, see \cite{KRS21} for more details  (see also \cite{STY01}, for conditional stability estimates for the inverse  Laplace transform that are analogous to the usual stability estimates of the Calderón problem). In the Hausdorff moment problem one can expect also to have logarithmic stability estimates. One (rough) way to look at this is through \eqref{est:moment_aprox_radial}: since the moments  $\sigma_k[q]$ and $\lambda_k[q] -k$ have the same size, the arguments of \cite{Mand00} and \cite{KRS21} should imply that the maps  $(\lambda_k[q]-k)_{k\in\N_0} \mapsto q$   and  $  (\sigma_k[q])_{k\in\N_0} \mapsto q$ have similar restrictions on possible conditional stability estimates.

As mentioned in the introduction,  it is natural to  ask if the hard and unstable part of decompressing the information from the DtN map is already been achieved by the formulas \eqref{id:gae_formula} and \eqref{id:radial_born}. This would be reflected in better stability estimates for the map $\qe \mapsto q$, or for the map    $(\lambda_k[q]-k)_{k\in\N_0} \mapsto (\sigma_k[q])_{k\in\N_0}$. The fact that the Born approximation in numerical experiments captures the high frequency part of the potential together with the  values  of $q$ close to the boundary,   offers some indirect evidence that this could be true. Also, the Born approximation approximates very well small potentials (say,  $\norm{q}_{L^\infty(B)} \le 2$) which also suggests  better stability  for $\qe \mapsto q$ with a smallness condition. From the point of view of the moments and the eigenvalues, \eqref{est:moment_aprox_radial_2}   shows that the largest differences between them appear for low values of $k$, which also could imply   stability estimates better than logarithmic for the map    $(\lambda_k[q]-k)_{k\in\N_0} \mapsto (\sigma_k[q])_{k\in\N_0}$. The same can be considered for the conductivity problem. All these questions remain open.

The instability of the moment problem is reflected in this work in the need of high precision computation of the values of $\lambda_k[q]-k$ in order to obtain accurate reconstruction of the high frequency part of $\qe$ and $\gae$. 
 In general to reconstruct $\qe$ and $\gae$  with \eqref{id:radial_born} and \eqref{id:gae_formula} up to frequencies  $|\xi| \le N$ one has to sum the series up to $k\sim 2N$. Thus, by \eqref{est:moment_aprox_radial}, the relevant data  is of order $\alpha^{2k}$ with $k \sim 2N$ if $q$ or $\gamma-1$ are supported in $B_\alpha\subset B$ (see more details in \Cref{sec:numerico}). 
 As a consequence, recovery of the high frequency parts of $\qe$ and $\gae$ with formulas \eqref{id:gae_formula} and \eqref{id:radial_born} using real data is a hard problem  (note, however, that this difficulty is also present in the full reconstruction problem).  On the other hand, the advantage of defining  $\qe$ as the function  that satisfies $\sigma_k[\qe] =  \lambda_k[q]-k$ is that one does not  need to apply  necessarily \eqref{id:radial_born}, but use instead  other numerical approaches to the moment problem.

As a consequence of the previous discussion, it is natural to expect that regularization techniques (see \cite{moment_theory_book}) commonly used to deal with noisy and real experiment data will be   necessary to obtain $\qe$ and $\gae$.
 The need of regularization techniques can be seen in \eqref{id:gae_formula} and \eqref{id:radial_born}: a small perturbation of one of the eigenvalues adds a derivative of a $\delta$  distribution  to $\gae$ and $\qe$.


\section{Numerical computation of the Born approximation} \label{sec:numerico}

In this section we describe some implementation details to obtain the Born approximation and include numerical experiments both for the potential and conductivity problems when $d=3$. 

\subsection{Computation of the DtN map}

Let $R=1$. The Born approximation formulas \eqref{id:gae_formula}  and \eqref{id:radial_born} and the reconstruction algorithms \eqref{eq:ga_itera_1} and \eqref{eq:itera_q} require very accurate values of   $\lambda_k[\gamma]-k$ in the case of the conductivity,   and $\lambda_k[q]-k$ in the case of the potential.
In the second case, we compute these eigenvalues using the recursive algorithm described in \cite{Fagueye} for piece-wise constant radial potentials. In the case of the conductivity we use the algorithm given in \cite{BKM11}, again for piece-wise constant conductivities. 

Due to the continuity of the DtN map with respect to the potential in the $L^\infty$ norm, we can approximate the DtN map of any continuous function from the one associated to a sufficiently close (in the $L^\infty$ norm) piece-wise constant function. 
We just take a uniform partition and approximate the conductivity (or the potential) by the function that takes, at each subinterval, the value at the middle point. 
In the examples below we have considered approximations with a uniform partition of up to 10,000 subintervals in $r\in[0,1]$. 

Another important fact is the computer precision. 
As we describe below we consider $400$ eigenvalues in our experiments.
However, since $\lambda_k[q]-k$ decays exponentially in $k$,  accurate approximations of these eigenvalues require more than the standard Float64 arithmetic precision. 
In our experiments we have considered up to Float1024 precision, according to the specific case.

\subsection{Implementation details in the Born approximation}

To compute the Born approximation we use formulas (\ref{id:radial_born}) for the potential $q$ and (\ref{id:gae_formula}) for the conductivity problem. 
This requires a discrete inversion formula for the Fourier transform of radial functions, and an accurate approximation of the corresponding series. We discuss both issues below.

The inverse Fourier transform of a radial function can be approximated with the one-dimensional discrete inverse  Fourier transform.  
In fact, if $f:\mathbb{R}^3 \to \mathbb{R}$ is a radial function and $f(x) = f_0(|x|)$, the Fourier transform is given by
\begin{eqnarray*}
\widehat{f} (\xi)&=& \mathcal F_3(f)(\xi) = \int_{\mathbb{R}^3 } f(x) e^{-ix\cdot \xi} dx \\
&=&\frac{(2\pi)^{3/2}}{|\xi|^{1/2}} \int_0^\infty r^{3/2}f_0(r)\frac{\sqrt{2}\sin(|\xi|r)}{\sqrt{\pi r|\xi|}}  \; dr
= \frac{4\pi}{|\xi|} \int_0^\infty r\; f_0(r)\sin(|\xi|r)  \; dr.
\end{eqnarray*}
Let $h: \R \to \R$ be the odd extension of $rf_0(r)$ to $\R$, that is $h(r) = rf_0(|r|)$.
Then
\[
  \mathcal F_3(f)(\xi) =\frac{2\pi}{|\xi|} \int_{-\infty}^\infty r\; f(|r|)\sin(|\xi|r)  \; dr=-\frac{2\pi}{|\xi|} \Im \left [\mathcal F_1 (h)(|\xi|) \right],
\]
where $\mathcal F_1$   and $\mathcal F_3$ stand, respectively,  the one dimensional and the three dimensional Fourier transforms.

Analogously, one can write the three dimensional inverse Fourier transform $\mathcal F_3^{-1}$ in terms of the one dimensional inverse Fourier transform $\mathcal F_1^{-1}$ as follows. Since $\widehat{f}$ is radial, we have that $\widehat{f}(\xi) = g_0(|\xi|)$ for some function $g_0:(0,\infty) \to \R$. Thus
\[
 f(x)=\mathcal F_3^{-1}   [\widehat{f}   ] (x)= -\frac{1}{(2\pi)^2 |x|} \Im \left[  \mathcal F_1^{-1} (h)(|x|) \right].
\]
where     $h: \R \to \R$ is again the   extension    $h(\rho) = \rho g_0(|\rho|)$.

In practice, we approximate a sampling of the function $f$ from a suitable sampling of its Fourier transform using the discrete inverse Fourier transform. To improve resolution we recover the extension of $q$ by zero to the interval $[0,10R]$, although we only draw the restriction to $[0,R]$ in the experiments below. 
Thus, a uniform sampling of $q$ in this interval with $N+1$ values is considered: $\{f(r_j)\}_{j=0}^N$ with $r_j=h j,$ $j=0,...,N$, $h=10R/N$. This requires a uniform sampling of $\hat f(\xi) $ in the interval $[0,\pi/(Rh)]$ with the same number of points, i.e. $\{ \hat f(\xi_j)\}_{j=0}^N$ with $\xi_j=h_\xi j,$ $j=0,...,N$, $h_\xi =\pi/(10R)$. 
As described in \cite[Lemma3]{BCR16} the convergence of this sampling to $f$ can be estimated in terms of the regularity of $f$ and its support, since an aliasing contribution appears for non-compactly supported functions in $B_R$. In our case, we are not able to establish precise error estimates, since we lack precise estimates on smoothness and support of the Born approximation.   

The second important issue is the approximation of the series, for example (\ref{id:radial_born}) in the case of the potential $q$. 
Obviously, we only sum a finite number of terms. 
As described in \cite{born_aprox} for any given $\xi$ the main contribution of this series is in the first terms $k\sim 2|\xi|$. 
We observed that for $\xi\in[0,160]$  the first $400$ terms in the series produce an stable approximation in the sense that adding more terms produces contributions of the order of $10^{-16}$ to the Fourier transform in this interval. 

On the other hand, note that the series in (\ref{id:radial_born}) contains very large terms for which the usual computer precision Float64 is not sufficiently good. 
In fact, this sum can only be computed accurately for $\xi\in[0,30]$ with the standard precision. To obtain $\xi\in[0,160]$ we considered Float1024 precision in our experiments. 

The codes are programmed with Julia which is very efficient with arbitrary precision computations. They can be downloaded from \cite{BCMM-P}.

\subsection{Numerical experiments: The conductivity case} \label{sec:3_gamma} \

  In this section we focus on the conductivity problem. We illustrate the main properties \ref{item:ga_1_int}--\ref{item:ga_4_int}  of the Born approximation with different experiments. In the first experiment we show that $\gae$ is defined even in the presence of jump discontinuities.  The second experiment illustrates the local uniqueness from the boundary, that is, property \ref{item:ga_1_int} in the introduction. Experiment 3 concerns the scattering limit of the Born approximation given by formula \eqref{id:gae_scattering}. Finally in experiment 4 we consider the Born approximation of different smooth conductivities: a conductivity close to $\gamma =1$, a very large conductivity and one degenerated example (a conductivity with reaches the value $0$ at the origin). 

\begin{figure}
    \centering
    \begin{tabular}{cc}
    \includegraphics[width=7cm]{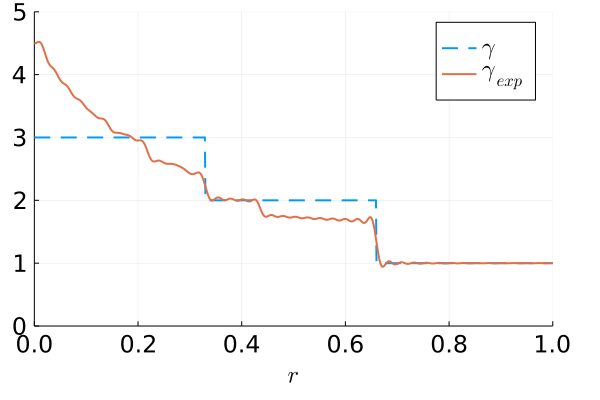} &
    \includegraphics[width=7cm]{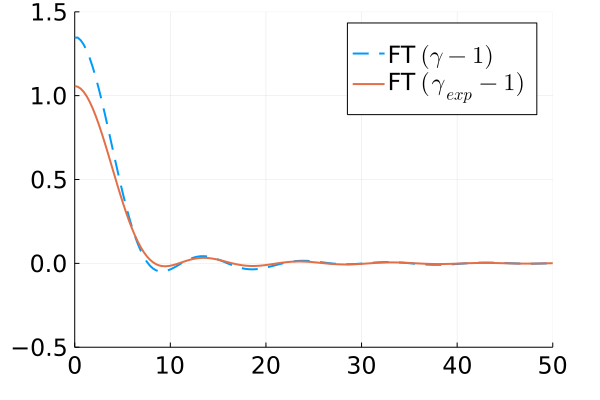} \\
    \includegraphics[width=7cm]{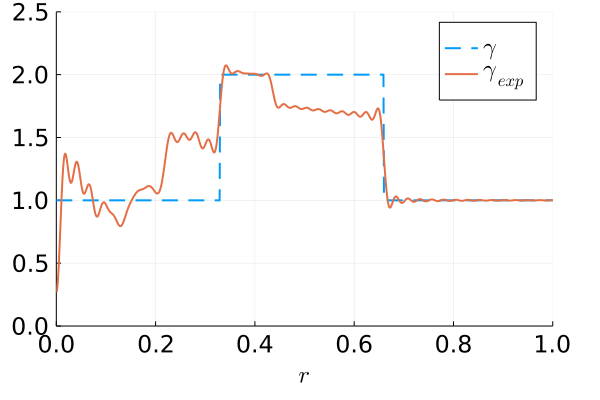} &
    \includegraphics[width=7cm]{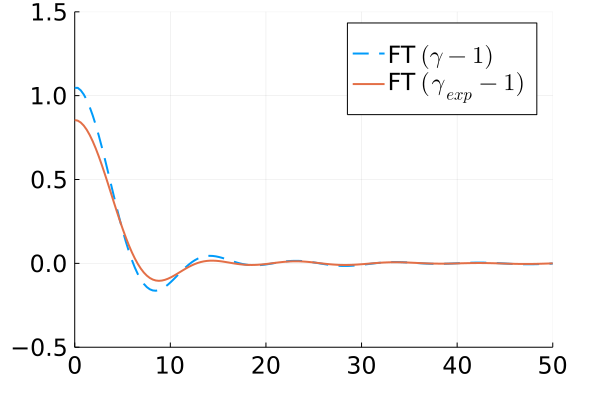}
    \end{tabular}
    \caption{Experiment 1: A step conductivity and its   Born approximation (left), and a comparison of their Fourier transforms (right). Lower simulations correspond to  a conductivity closer to the reference conductivity $\gamma=1$.}
    \label{fig:ex7}
\end{figure}

{\bf Experiment 1.} We first consider two piece-wise constant conductivities (see \Cref{fig:ex7}). This shows  that the Born approximation is well defined even for conductivities with jump discontinuities. The lower simulation provides a better approximation of the potential since it is closer to the    reference conductivity $\gamma=1$, which illustrates property \ref{item:ga_3_int} in the introduction. Notice also that  property   \ref{item:ga_4_int} is verified in both examples.

 {\bf Experiment 2.}  Here we illustrate the remarkable property   \ref{item:ga_1_int} of the Born approximation. In particular we observe that the Born approximations of $3$ different conductivities that coincide in an interval $(b,1)$, also  coincide    in $(b,1)$ (see \Cref{fig:ex_2}).  Notice that the Born approximations are much smoother than in the previous case, since there are no jump discontinuities in the conductivities.  This experiment also shows that the Born approximation is better close to the boundary |property \ref{item:ga_2_int}| and is a better approximation for conductivities closer to $\gamma=1$ |property \ref{item:ga_3_int}.  
 
 \begin{figure}
    \centering
    \begin{tabular}{cc}
    \includegraphics[width=7cm]{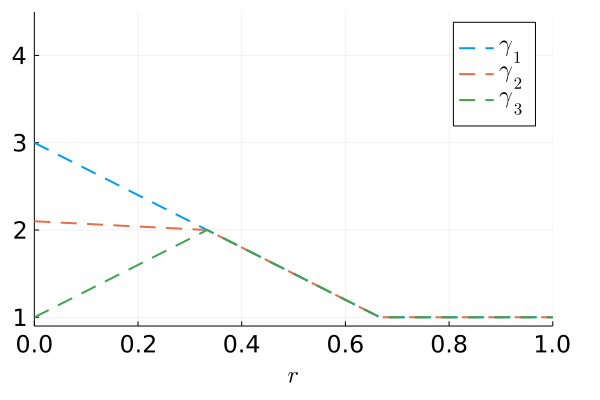} &
    \includegraphics[width=7cm]{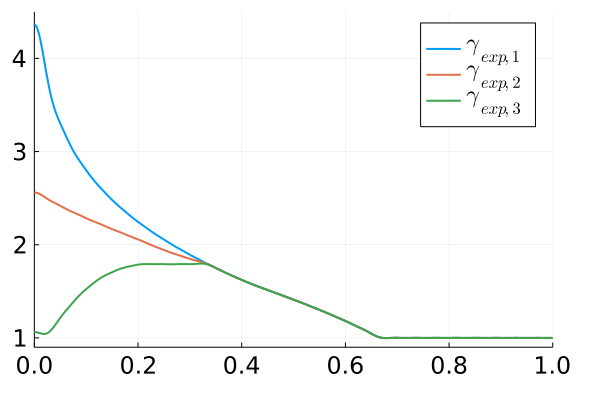} 
    \end{tabular}
    \caption{Experiment 2: Three different conductivities that coincide in the interval $(1/3,1)$ (left) and their Born approximations (right). We observe that they also coincide in this same interval $(1/3,1)$.}
    \label{fig:ex_2}
\end{figure}

{\bf Experiment 3.} 
 Here we  compare the Born approximation $\gae $ given by   \eqref{id:gae_formula} with $R =1$, and the scattering limit $\gae(\centerdot,\infty)$ given by formula \eqref{id:gae_scattering} (see \Cref{fig:ex12c}). We take 
\[
\gamma(r) = 2-\mathrm{sign(r-1/2)}(r-1/2).
\]
 We see that the scattering limit of the Born approximation deteriorates slightly with respect to $\gae$ but still recovers a fairly good approximation, in particular close to the boundary (in fact both  $\gae$ and $\gae(\centerdot,\infty)$  seem to satisfy property  \ref{item:ga_2_int}).

\begin{figure}
    \centering
    \begin{tabular}{cc}
    \includegraphics[width=7cm]{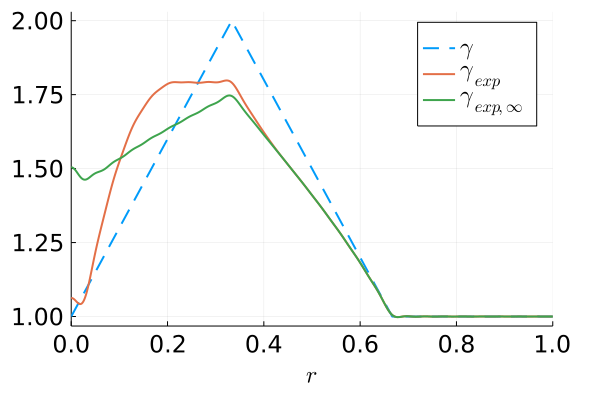} &
    \includegraphics[width=7cm]{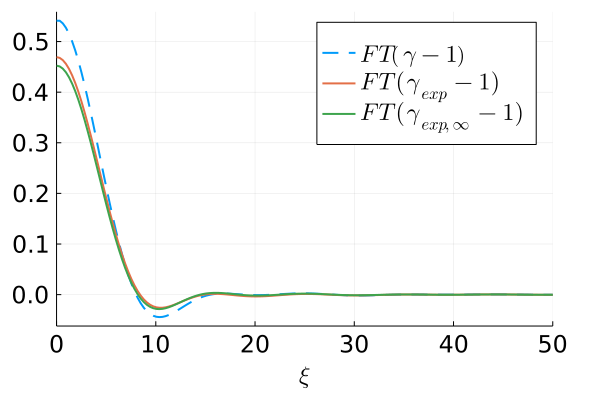} 
    \end{tabular}
    \caption{Experiment 3: Scattering limit. A conductivity   $ \gamma$, its  Born approximation $\gae $ given by   \eqref{id:gae_formula} with $R =1$, and the scattering limit $\gae(\centerdot,\infty)$ given by formula \eqref{id:gae_scattering} (left).}
    \label{fig:ex12c}
\end{figure}

{\bf Experiment 4.} 

Here we consider the Born approximation for smooth conductivities.
The simulations in the first row in \Cref{fig:ex10} correspond to the example considered in \cite{BKM11} where the Fourier transform was computed only for $|\xi| < 30$ due to numerical instabilities. Here we have circumvented these instabilities with higher machine precision. 
As illustrated by the second row in \Cref{fig:ex10}, the Born approximation deteriorates as $\gamma$ becomes very large, which is something to be expected since $\gae$ is constructed from a linearization at $\gamma =1$.
The singular case illustrated in the third row of \Cref{fig:ex10} corresponds to a conductivity which is close to zero near $r=0$. Observe that we still have a good approximation close to the boundary, but near the origin the Born approximation degenerates to negative values.

\begin{figure}
    \centering
    \begin{tabular}{cc}
    \includegraphics[width=7cm]{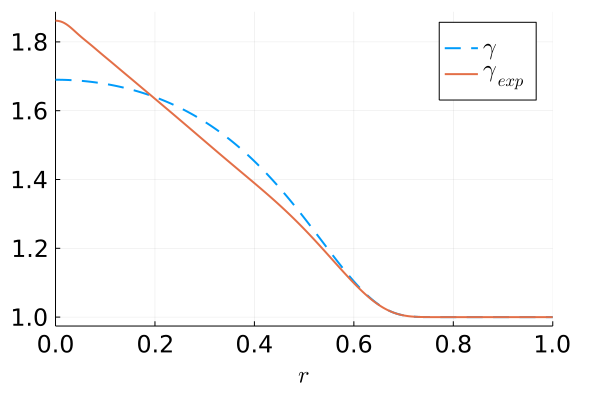} &
    \includegraphics[width=7cm]{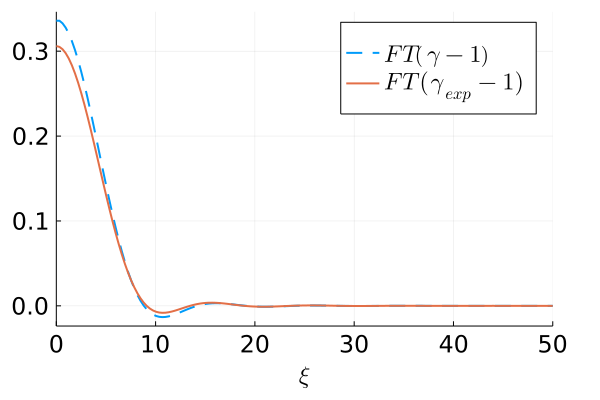} \\
    \includegraphics[width=7cm]{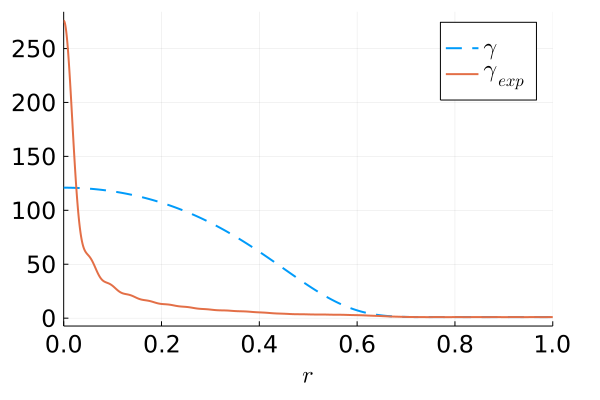} &
    \includegraphics[width=7cm]{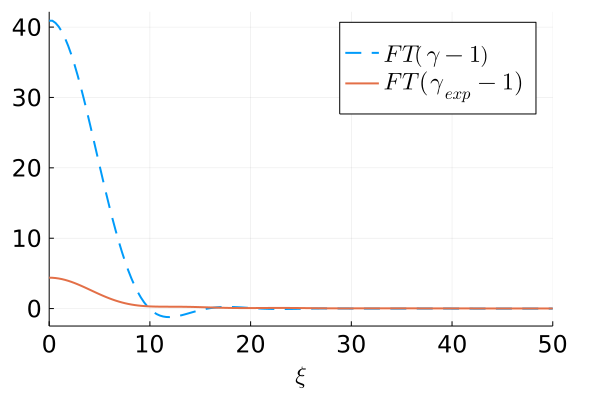} \\
    \includegraphics[width=7cm]{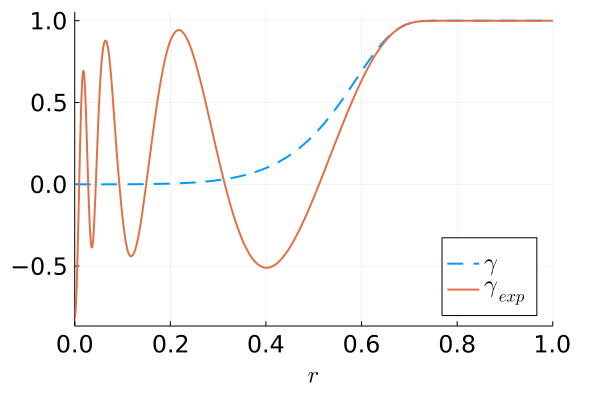} &
    \includegraphics[width=7cm]{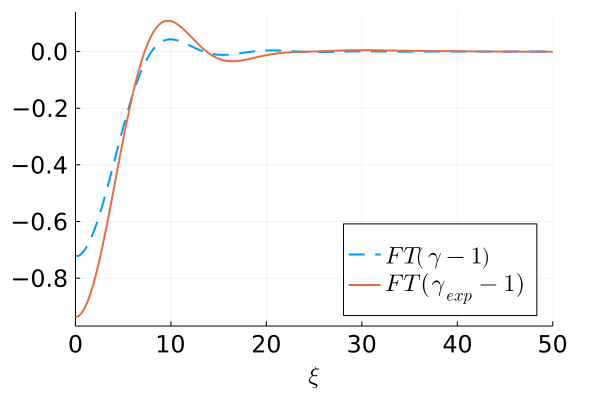} \\
    \end{tabular}
    \caption{Experiment 4: Born approximations (left) and their Fourier transform (right) for smooth conductivities with different sizes.}
    \label{fig:ex10}
\end{figure}

\bigskip

\subsection{Numerical experiments: The potential case} \label{sec:3_potential} 

 In this section we focus on the potential case and again analyze the Born approximation through  different experiments. The aim is to illustrate properties  \ref{item:q_1}-- \ref{item:q_5} of \Cref{sec:2} and other similar phenomena.  Experiments 5  and 6 illustrate the  Born approximation for potentials of different sizes and regularities. Experiment 7 explores the depth dependence of the accuracy of the Born approximation by considering the approximation of a large set of randomly chosen potentials. In this way we can make a quantification of property \ref{item:q_3}.
In Experiment 8 below, we show that, when the potential is zero in a sufficiently large neighborhood of $r=0$, the    Born approximation is able to recover this value. Experiment 9 analyzes how the Born approximation of a fixed potential changes for different values of $R$ and in the scattering limit $R\to \infty$.
Finally, in Experiment 10 we illustrate a particular oscillatory behavior associated to large negative potentials.

\bigskip

{\bf Experiment 5.} Here we consider piece-wise constant radial potentials with two steps in the unit ball $B$. We first compute the eigenvalues of the DtN map $\{\lambda_k[q]\}_{k\in\N_0}$ as described previously. From this sequence we have obtained  the Born approximation according to formula \eqref{id:radial_born}.    In \Cref{fig:ex1} and \Cref{fig:ex1b} we compare the Born approximation and the potential, and their Fourier transforms.

In the first row of  \Cref{fig:ex1} and the first two rows of \Cref{fig:ex1b} we consider step potentials with $\norm{q}_{L^\infty}\le 10$. In these examples the Born approximation is a much better approximation close to the boundary of the ball than close to the origin, which illustrates property \ref{item:q_3}. This property is also satisfied in the remaining cases, though   the approximation deteriorates when the size of the potential increases. In particular, we observe that oscillations appear close to $r=0$ |property \ref{item:q_5}| for very large potentials, see the last rows of \Cref{fig:ex1,fig:ex1b}. 

We observe also a very clear recovery of singularities phenomenon |property \ref{item:q_4}. All the discontinuities of the step potentials are also present in the respective Born approximations, even for very large potentials.
This is also appreciated in the Fourier transforms: large differences between the Fourier transform and the potentials are produced only for low frequencies. 
\begin{figure}
    \centering
    \begin{tabular}{cc}
    \includegraphics[width=7cm]{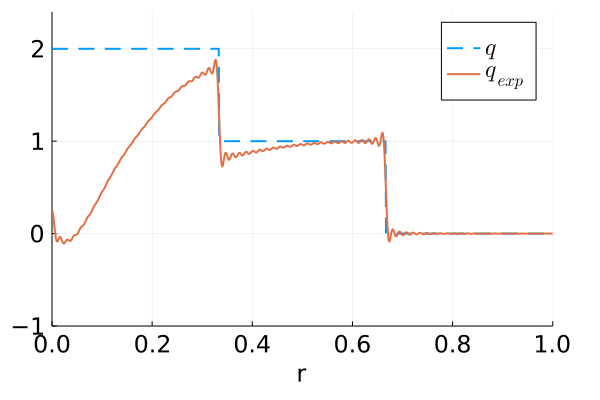} &
    \includegraphics[width=7cm]{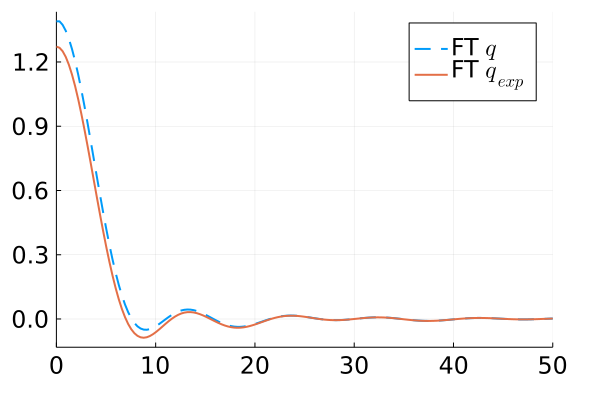} \\
    \includegraphics[width=7cm]{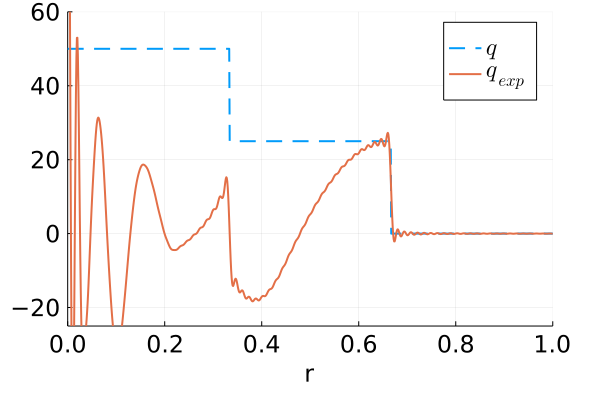} &
    \includegraphics[width=7cm]{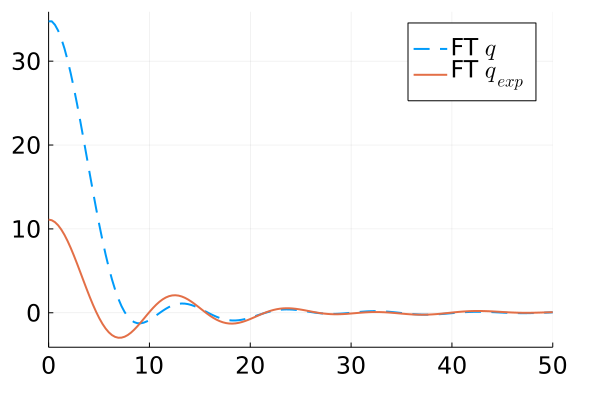} \\
    \includegraphics[width=7cm]{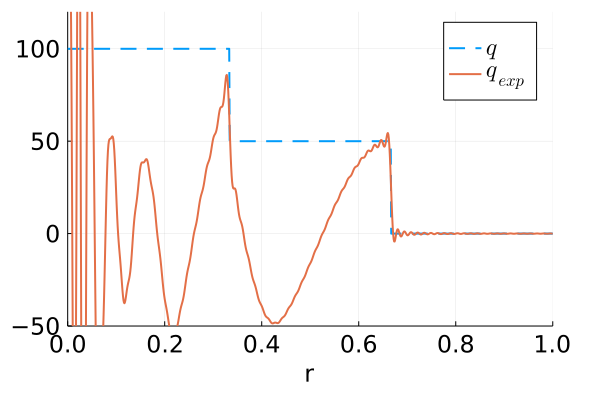} &
    \includegraphics[width=7cm]{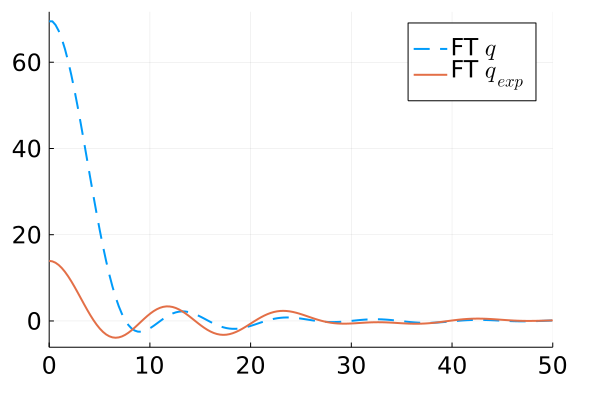} \\
    Potential & Fourier transforms
    \end{tabular}
    \caption{Experiment 5: Born approximation of  increasingly larger step potentials (left) and its Fourier transforms (right).}
    \label{fig:ex1}
\end{figure}
In the particular  case of \Cref{fig:ex1b}, we consider a bump type potential that vanishes in a neighborhood of $r=0$. This fact seems to  improve the behaviour of the Born approximation close to $r=0$ for medium and small size potentials. This phenomena will  also be illustrated for a smooth potential (see experiment 8 below).
\begin{figure}
    \centering
    \begin{tabular}{cc}
    \includegraphics[width=7cm]{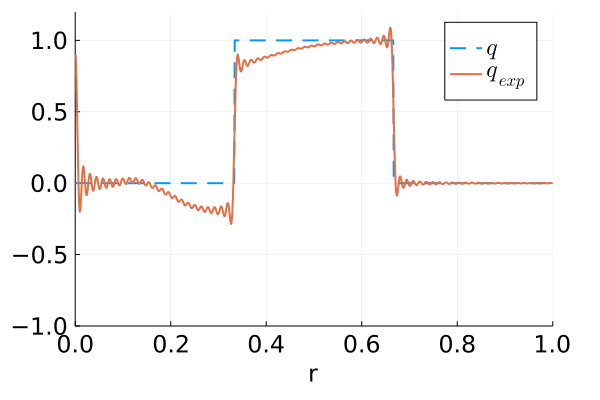} &
    \includegraphics[width=7cm]{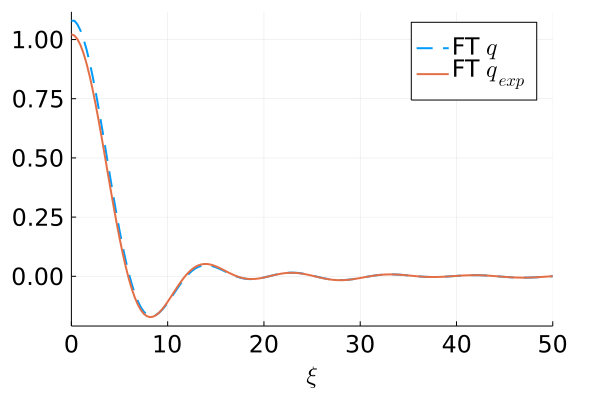} \\
    \includegraphics[width=7cm]{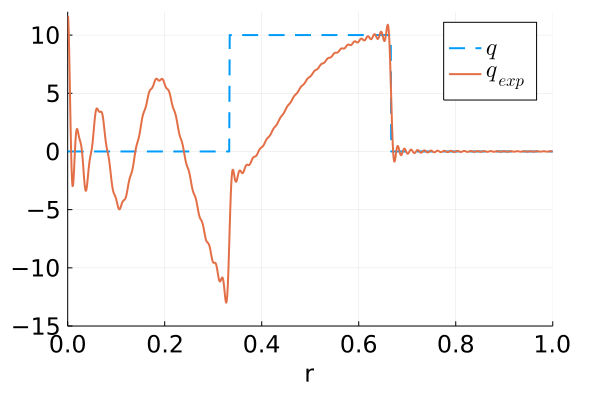} &
    \includegraphics[width=7cm]{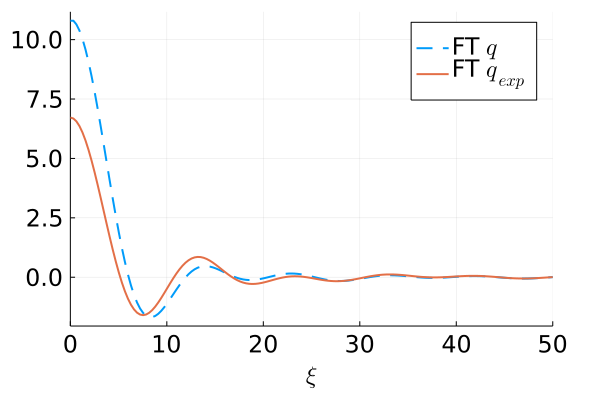} \\
    \includegraphics[width=7cm]{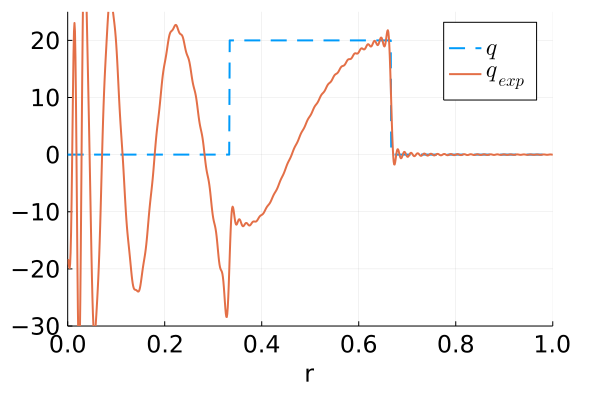} &
    \includegraphics[width=7cm]{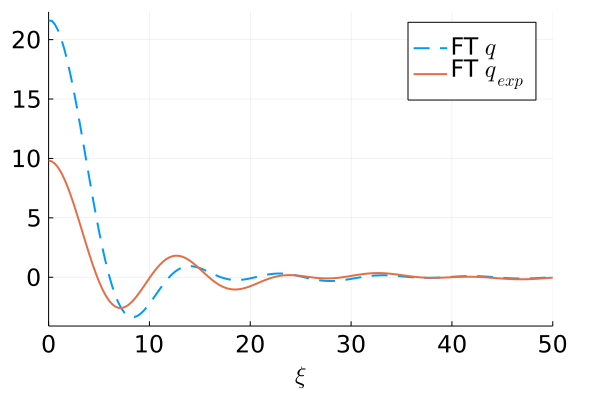} \\
    Potential & Fourier transforms
    \end{tabular}
    \caption{Experiment 5: Born approximation of a discontinuous bump potential and its Fourier transform}
    \label{fig:ex1b}
\end{figure}

\bigskip

{\bf Experiment 6.} This example (see \Cref{fig:ex_7}) is similar but now we consider a smooth function. In this case we do not have a explicit formula for the eigenvalues of the DtN map and we use an approximation with a piecewise constant function as described above. Again, the Born approximation is fairly good  for a small potential (first row) and  for a medium size potential it is a better approximation  closer to the boundary  (middle row) than close to the origin, which illustrates property \ref{item:q_3} in the smooth case. 
Again, for a large potential  the low frequencies separate substantially at low frequencies (last row, right)   and oscillations appear near $r=0$.

\begin{figure}
    \centering
    \begin{tabular}{cc}
    \includegraphics[width=7cm]{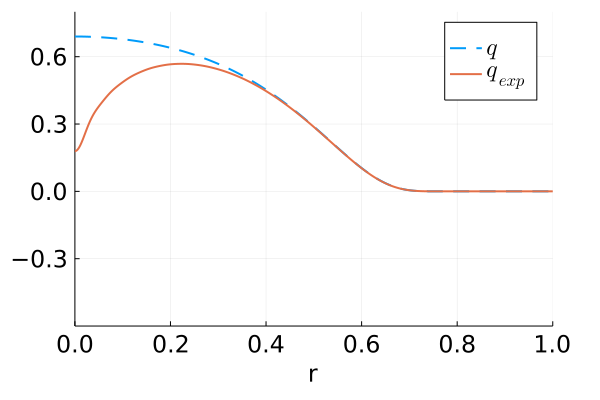} &
    \includegraphics[width=7cm]{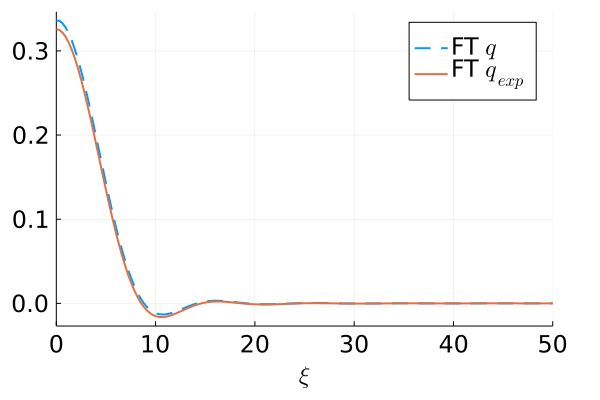} \\
    \includegraphics[width=7cm]{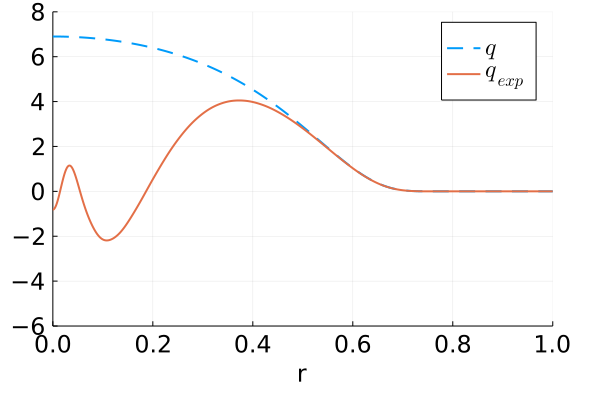} &
    \includegraphics[width=7cm]{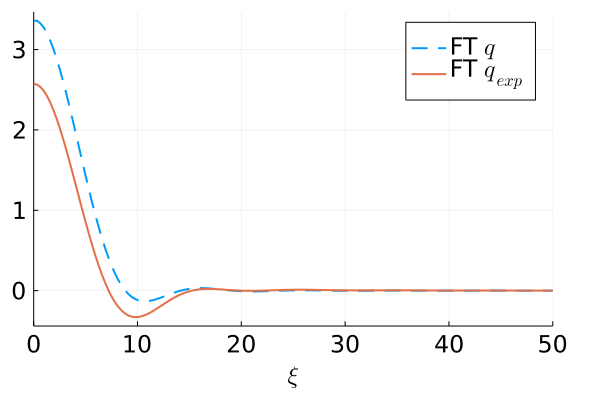} \\
    \includegraphics[width=7cm]{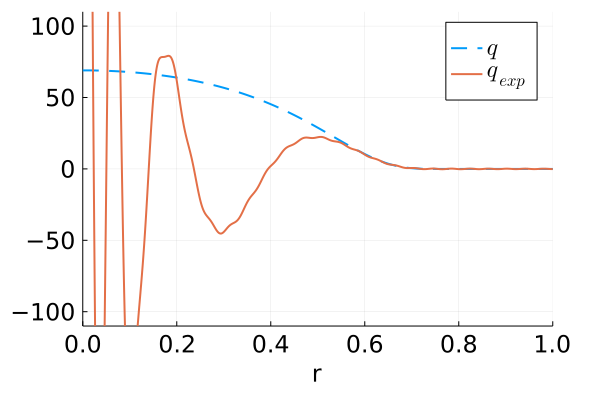} &
    \includegraphics[width=7cm]{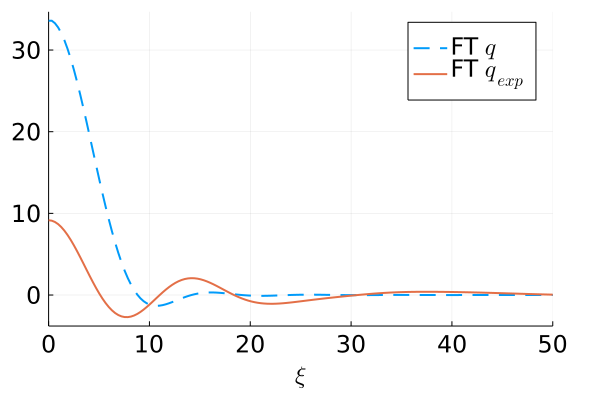} \\
    Potential & Fourier transforms
    \end{tabular}
    \caption{Experiment 6: Born approximation of a smooth potential and its Fourier transform}
    \label{fig:ex_7}
\end{figure}

\bigskip

{\bf Experiment 7.} From the previous two examples we see that the Born approximation is more accurate close to the boundary. This is a general issue as illustrated in the next example (see \Cref{fig:ex_8}) where we have computed the error $e(r)=\frac1{N_s}\sum_{i=1}^{N_s}| q(r)-\qe (r)|$ for a random sampling of $N_s=100$ potentials. We have chosen the potentials as linear combination of the first 20 trigonometric basis functions $\{ \sqrt{2} \cos(\pi(j-1/2)r)\}_{j\geq 1}$ which are orthonormal in $L^2$ and satisfy the boundary conditions $\varphi'(0)=\varphi(1)=0$, i.e.
\[
q(x)=\sum_{j=1}^{20} c_j \sqrt{2} \cos(\pi(j-1/2)r).
\]
The Fourier coefficients $c_j$ are chosen randomly in the interval $c_j\in [-1/j,1/j]$ in such a way that
$\| q\|_{L^2}^2 = \sum_{j=1}^{20} |c_j|^2\leq 1$. We observe that the error $e(r)$ decrease as $r\to 1$. We also illustrate the behavior of this error for larger potentials of the form $q^\alpha(r)=\alpha q(r)$, $\alpha=2,3$. In this case, the relative error $e_\alpha=e/\alpha$, where $e(r)=\frac1{N_s}\sum_{i=1}^{N_s}| q^\alpha(r)-[q^\alpha]_{\mathrm{exp}}(r)|$, is close to zero near $r=1$ but becomes larger as $r$ approaches to $r=0$.   This is essentially a quantification of property \ref{item:q_3}.
\begin{figure}
    \centering
    \includegraphics[width=7cm]{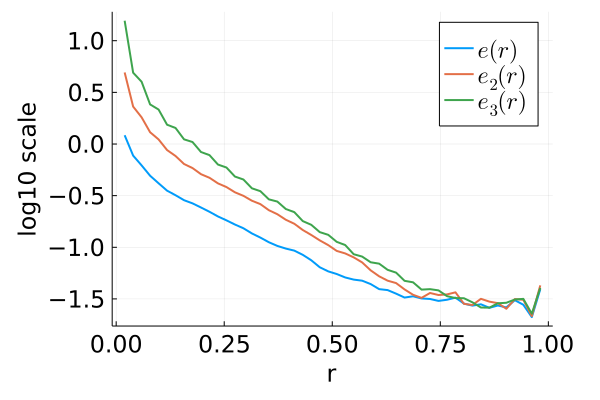} 
    \caption{Experiment 7: average error distribution of the Born approximation, $e_\alpha(x)$, computed with 100 samples, where Fourier coefficients are chosen randomly in different intervals}
    \label{fig:ex_8}
\end{figure}

\bigskip

{\bf Experiment 8.} An interesting feature appears when the potential is zero near $r=0$. In this case, the Born approximation is somehow able of recovering this value as illustrated in the example in \Cref{fig:ex4} (see also \Cref{fig:ex1b}), at least if the potential is not too large. This seems to be very specific of the zero value since the error is in general much larger when the potential is not zero near $r=0$, as illustrated in the previous experiment. 

\begin{figure}
    \centering
    \begin{tabular}{cc}
    \includegraphics[width=7cm]{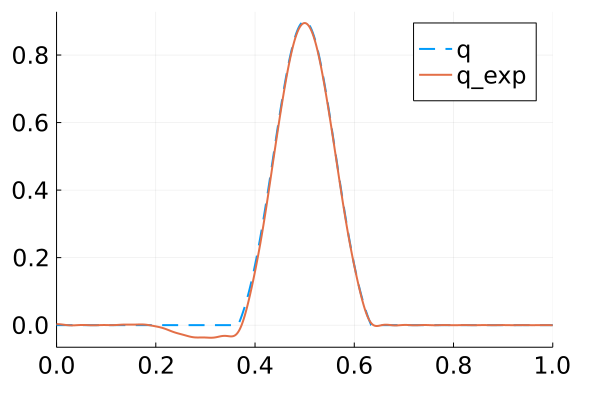} &
    \includegraphics[width=7cm]{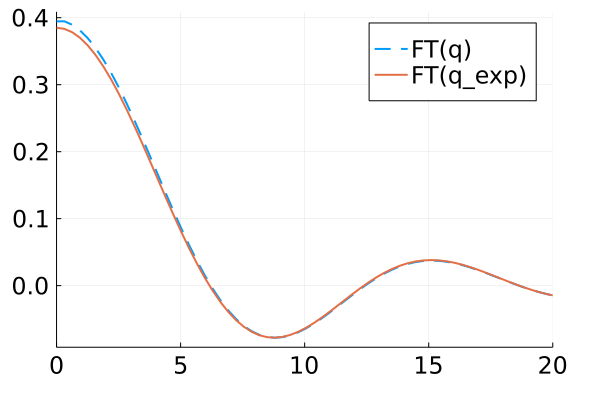}
    \end{tabular}
    \caption{Experiment 8: Born approximation of a potential which is zero in a neighborhood of $r=0$ (left) and its Fourier transform (right).}
    \label{fig:ex4}
\end{figure}

\bigskip

{\bf Experiment 9.} In this example we consider the same potential in larger domains, i.e. we take as domain $B_R$ for $R=5,\infty$ (the case $R=1$ is the same as in \Cref{fig:ex_7}). We observe in \Cref{fig:ex5} that the Born approximation, given by \eqref{id:qexp_formula_R_2}, deteriorates but maintains a good approximation near $r=1$. The case $R=\infty$ corresponds to the scattering limit given by formula (\ref{id:qexp_formula_R_3}). 

\begin{figure}
    \centering
    \begin{tabular}{cc}
    \includegraphics[width=7cm]{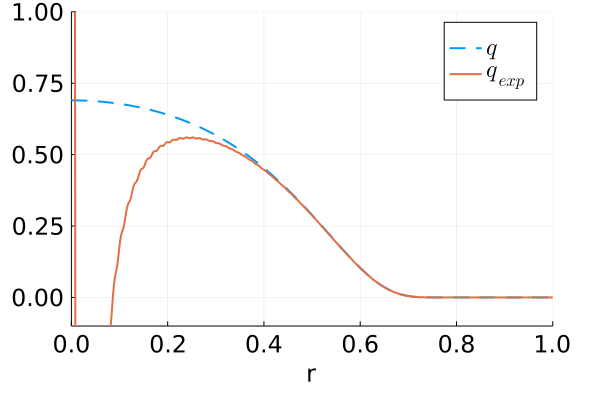} &
    \includegraphics[width=7cm]{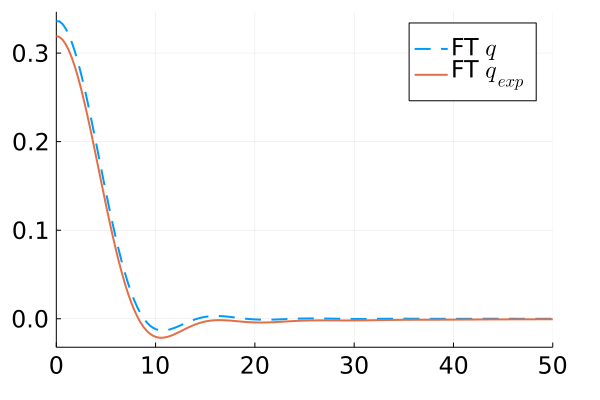}\\
        \includegraphics[width=7cm]{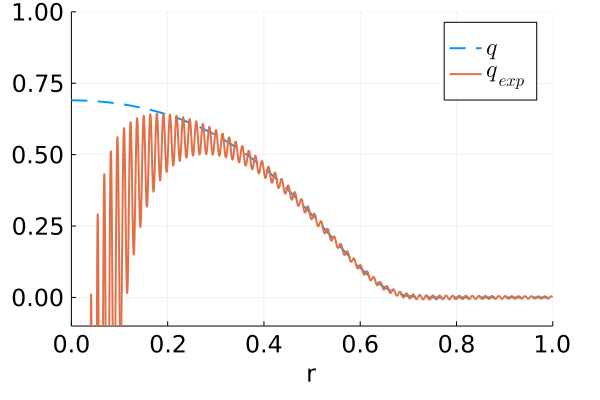} &
    \includegraphics[width=7cm]{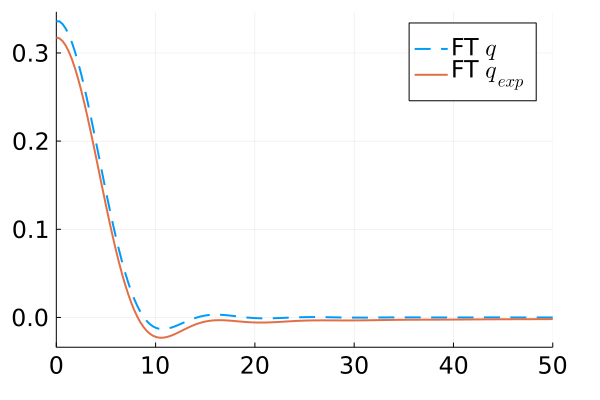}
    \end{tabular}
    \caption{Experiment 9: Born approximation of a potential  (left) and its Fourier transform (right) when we consider domains with $R=5$ (upper figure), $R=\infty$ (bottom one).}
    \label{fig:ex5}
\end{figure}

\bigskip

{\bf Experiment 10.} When the potential is negative the Born approximation seems to be less stable as shown in the experiment illustrated in \Cref{fig:ex6}. As we increase the size of the potential Fourier transform of the Born approximation and the potential separate. This produces both oscillations and a singularity at $r=0$ in the Born approximation.  

\begin{figure}
    \centering
    \begin{tabular}{cc}
    \includegraphics[width=7cm]{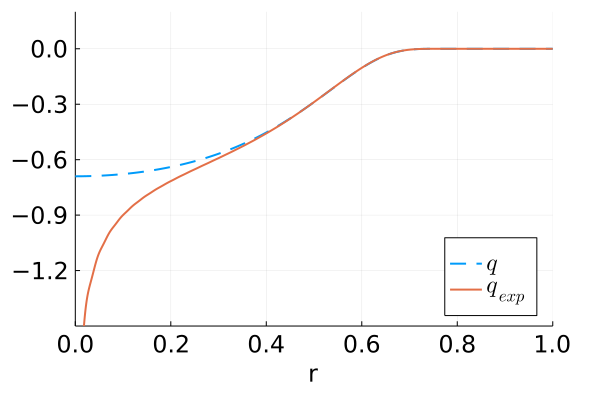} &
    \includegraphics[width=7cm]{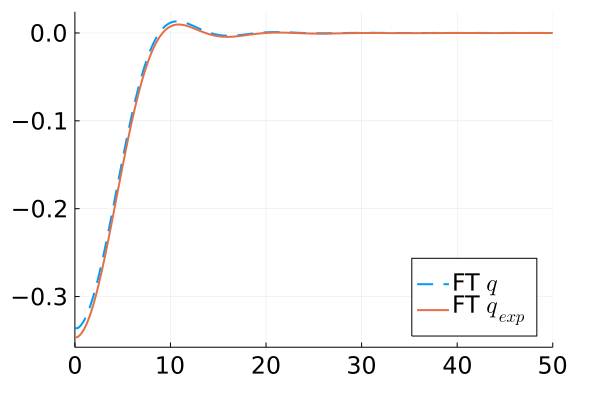}\\
        \includegraphics[width=7cm]{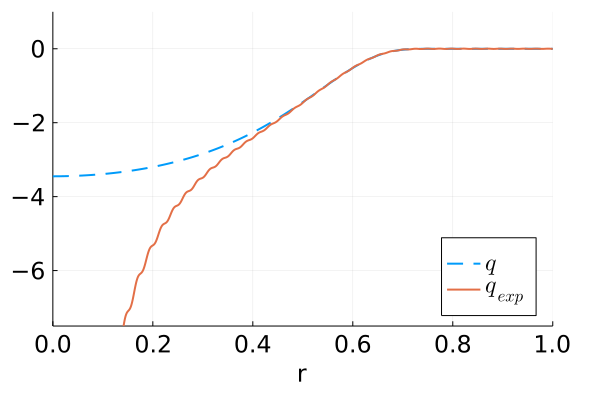} &
    \includegraphics[width=7cm]{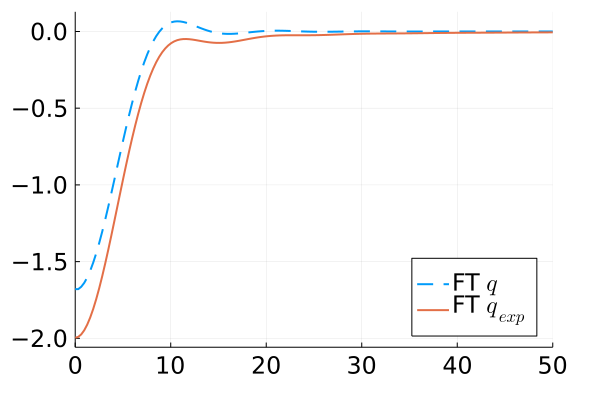}\\
        \includegraphics[width=7cm]{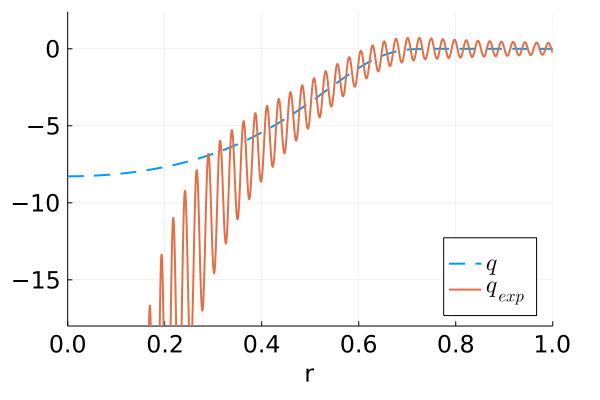} &
    \includegraphics[width=7cm]{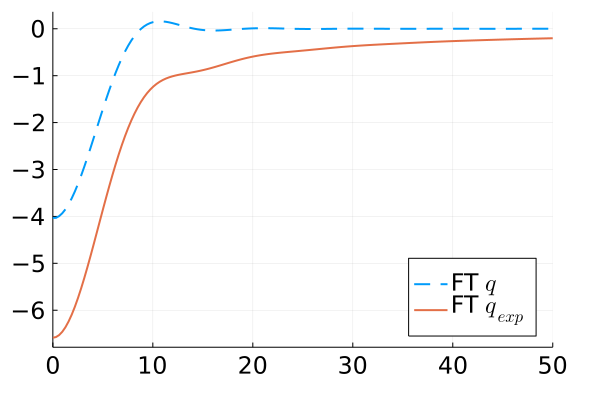}
    \end{tabular}
    \caption{Experiment 10: Born approximation of a potential  (left) and its Fourier transform (right) when we consider negative potentials with larger size (from top to bottom).}
    \label{fig:ex6}
\end{figure}

\section{Iterative algorithm} \label{label:algorithm}

In this section we illustrate the efficiency of the iterative algorithm described in \eqref{eq:ga_itera_1} for the conductivity and in \eqref{eq:itera_q} for the potential.

{\bf Experiment 11.} In \Cref{fig:ex_it11} we show the first iterations of the algorithm described in \eqref{eq:ga_itera_1} when considering a Lipschitz conductivity and a smooth one. At each iteration the conductivity is better approximated. The behavior of the $L^2$-error and $L^\infty$-error are also illustrated in  \Cref{fig:ex_it11_b}. We observe that the rate of convergence depends on the regularity of $\gamma$. Note also that the piecewise constant conductivity considered in Experiment 1 above (in green-star) produces a decreasing error only for the $L^2-$norm. This is due to the Gibbs phenomenon which is present since we only compute a low pass filter of the conductivity.     

\begin{figure}
    \centering
    \begin{tabular}{cc}
    \includegraphics[width=7cm]{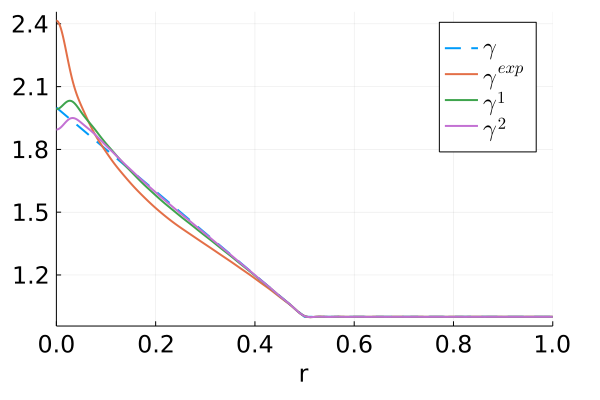} &
    \includegraphics[width=7cm]{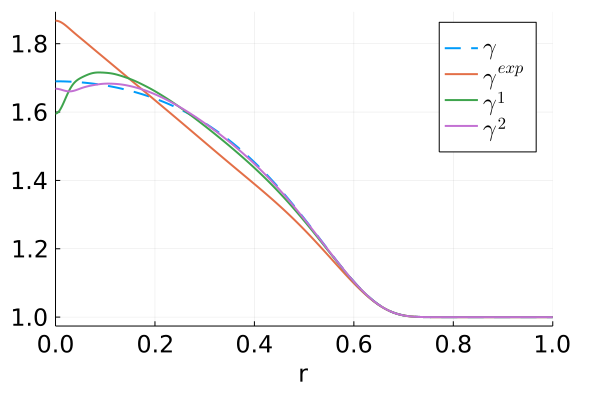} 
    \end{tabular}
    \caption{Experiment 11: Approximation of   a Lipschitz conductivity (left simulation) and a smooth one (right simulation) by the iterative algorithm \eqref{eq:ga_itera_1}.}
    \label{fig:ex_it11}
\end{figure}

\begin{figure}
    \centering
    \begin{tabular}{cc}
    \includegraphics[width=7cm]{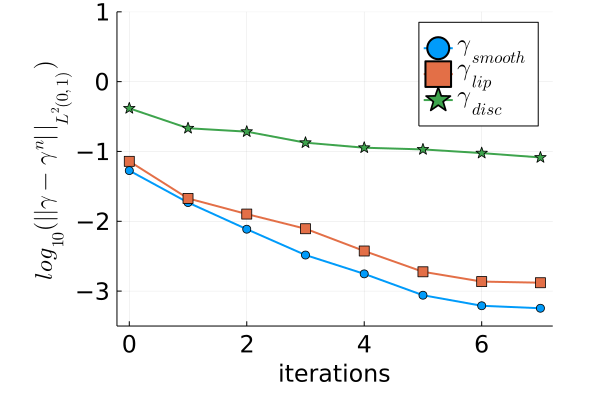} &
    \includegraphics[width=7cm]{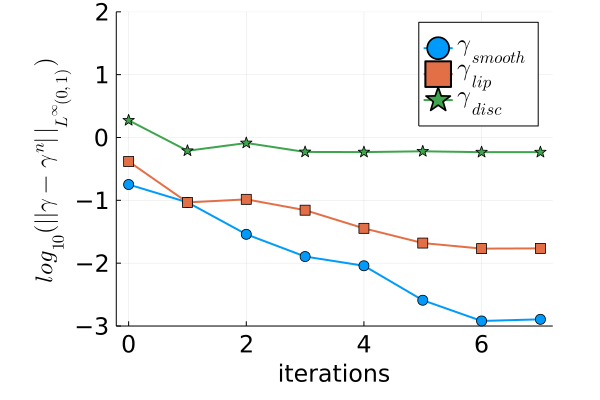} 
    \end{tabular}
    \caption{Experiment 11: $L^2$-error (left) and $L^\infty$-error (right) in $log_{10}$-scale for the iterative algorithm  when considering the piecewise constant function in the upper simulation of \Cref{fig:ex7}, together with the Lipschitz  and smooth  conductivities  of \Cref{fig:ex_it11}.}
    \label{fig:ex_it11_b}
\end{figure}

{\bf Experiment 12.} In \Cref{fig:ex_it12} we show the first iterations of the algorithm described in \eqref{eq:itera_q} when considering a discontinuous potential and a smooth one. At each iteration the approximation improves in a increasingly larger set. The behavior of the $L^2$-error and $L^\infty$-error are also illustrated in \Cref{fig:ex_it12_b}. We observe that the smoothness of the potential affects to the error behavior. In fact, for the discontinuous potentials it stabilizes after a few iterations. 

\begin{figure}
    \centering
    \begin{tabular}{cc}
    \includegraphics[width=7cm]{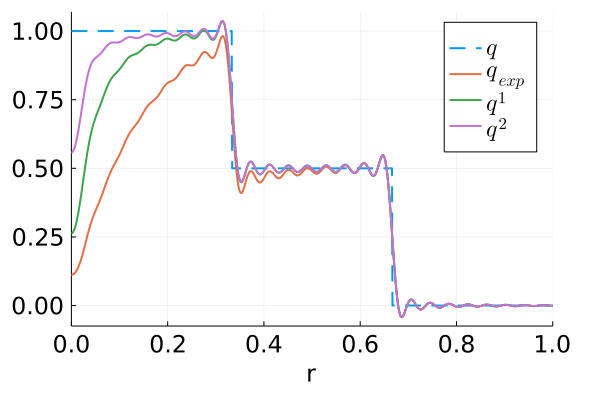} &
    \includegraphics[width=7cm]{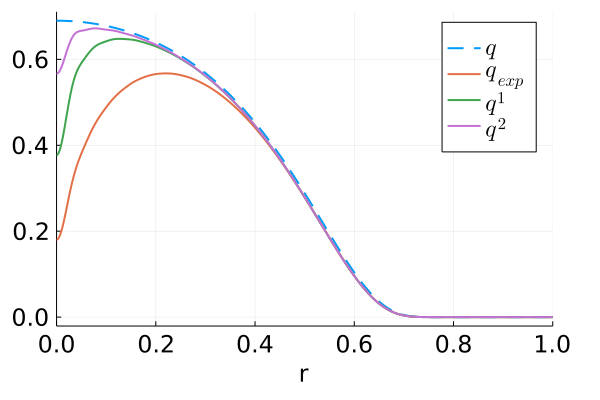} 
    \end{tabular}
    \caption{Experiment 12: Approximation of   a step potential (left simulation) and a smooth one (right simulation) by the iterative algorithm \eqref{eq:itera_q}.}
    \label{fig:ex_it12}
\end{figure}

\begin{figure}
    \centering
    \begin{tabular}{cc}
    \includegraphics[width=7cm]{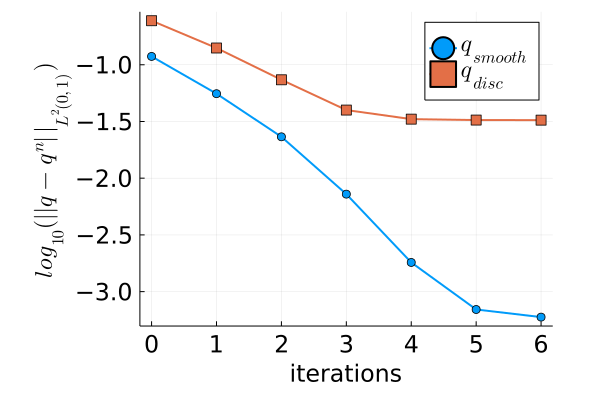} &
    \includegraphics[width=7cm]{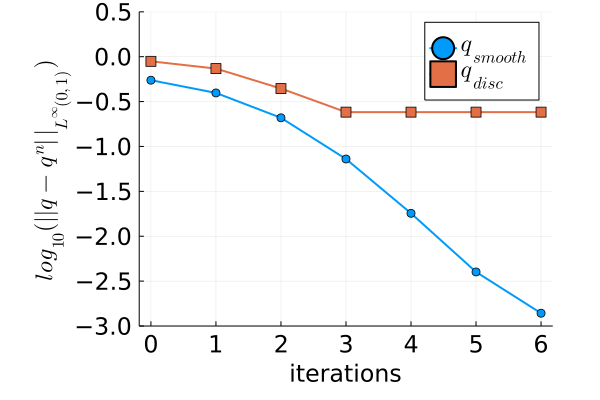} 
    \end{tabular}
    \caption{Experiment 12: $L^2$-error (left) and $L^\infty$-error (right) in $log_{10}$-scale for the iterative algorithm  when considering both the smooth potential and the discontinuous potential of \Cref{fig:ex_it12}.}
    \label{fig:ex_it12_b}
\end{figure}

\clearpage

\appendix
\section{Support of the Born approximation}   \label{sec:appendix_support}

Numerical results in \Cref{sec:numerico} suggest that $\gae -1$ and $\qe$ are supported always in the same ball as, respectively, $\gamma-1$ and $q$. 
As far as we know this is a property that does not hold  in  inverse scattering problems where the Born approximation is commonly used. 
In the rest of this section we show a proof of this property provided that one assumes that  formulas \eqref{id:radial_born} or \eqref{id:gae_formula} give a  tempered distribution.
 
We start by stating the following version of Paley-Wiener theorem.
\begin{proposition} \label{prop:paley_wiener}
Let $F \in \mathcal  S'(\R^d)$ be  a tempered distribution such that its Fourier transform has an holomorphic extension $\widetilde{F}(z)$ to all $\IC^d$. Assume also that
\[
|\widetilde{F}(z)| \le  C e^{\alpha|z|},
\] 
for constants $\alpha, C>0$. Then $F$ is supported in $B_\alpha \subset \R^d$, the ball of radius $\alpha$.
\end{proposition}
This version of Paley Wiener follows   from \cite[Theorem 4.9]{steinweiss} by using an appropriate mollification of the tempered distribution $F$. 

We now reduce to the case $d=3$ for simplicity, since the higher dimensional cases follow with minimal modifications. Define the following holomorphic function in $\IC^3$:
\begin{equation} \label{id:qexp_zeta}
\widetilde{\qe}(z) =2 \pi^{3/2}  \sum_{k=0}^\infty  \frac{  \left(z\cdot z\right)^{k}  2^{-2k}}{k! \Gamma(k+3/2)}
 (\lambda_{k}[q] -k),
\end{equation}
where the series is absolutely convergent by \eqref{est:moment_aprox_radial}. The function $\widetilde{\qe}(z)$ is the holomorphic extension of $ \widehat{\qe}(\xi)$ from $\R^3$ to   $\IC^3$.
\begin{lemma} 
Let $q \in\cQ_d^1$ be a radial potential supported in $B_\alpha \subseteq B$.  Then
\begin{equation} \label{est:analytic_est}
|\widetilde{\qe}(z)| \lesssim    e^{\alpha|z|}  ,
\end{equation}
where the implicit constants depends on $q$.
\end{lemma}
\begin{proof}
By \eqref{est:moment_aprox_radial} we have   for all $k\ge 0$ that
\begin{equation} \label{est:lambda_k_dif}
    |\lambda_{k}[q]-k|\lesssim    \frac{\alpha^{2k}}{2k+3},
\end{equation}
where the implicit constant depends on $q$. 
On the other hand, we have the identity
\begin{equation} \label{id:Gamma}
k! \Gamma\left(k + 1/2 \right) = \sqrt{\pi} (2k)! 2^{-2k},
\end{equation}
which implies that
\[
\frac{\sqrt{\pi}}{k! \Gamma\left(k + 3/2 \right)} 2^{-2k} = \frac{2}{(2k+1)} \frac{1}{(2k)!} .
\]
 Using this together with \eqref{est:lambda_k_dif} and  \eqref{id:qexp_zeta},  we obtain that
\begin{align*}
|\widetilde{\qe}(z)| &\lesssim  \sum_{k=0}^\infty  \frac{ |z|^{2k}  2^{-2k}}{k! \Gamma(k+3/2)}
     |\lambda_k[q] - k  | \\
&\lesssim     \sum_{k=0}^\infty \frac{1}{(2k+3)^2} \frac{\,|\alpha z|^{2k}   }{(2k)!}  \le   e^{\alpha|z|}.
\end{align*}
This finishes the proof of the lemma.
\end{proof}
Combining \Cref{prop:paley_wiener} with the previous lemma, we immediately obtain the following result.
\begin{proposition}
Let $q \in\cQ_d^1$ be a radial potential supported in $B_\alpha \subseteq B$.  Assume that the the function given by  \eqref{id:radial_born}  is a tempered distribution. Then $\qe$ is     supported   in $B_\alpha$.
\end{proposition}
Analogously one can prove the following proposition for the conductivity problem.
\begin{proposition}
Let $\gamma \in W^{2,\infty}(B,\R_+)$  be a radial conductivity such that $\gamma-1$ is supported in $B_\alpha \subseteq B$.  Assume that the function given by  \eqref{id:gae_formula}  is a tempered distribution. Then $\gae-1$ is     supported   in $B_\alpha$.
\end{proposition}
\begin{proof}
It follows from \Cref{prop:paley_wiener} and an estimate analogous to \eqref{est:analytic_est} for the analytic extension of $\widehat{\gae-1}(\xi)$.
\end{proof}

\bibliographystyle{alp}

\bibliography{references_calderon_numerico}

\end{document}